\documentclass[11pt]{article}
\usepackage{latexsym, amssymb, enumerate, amsmath, amsthm,pdfsync, multicol,numbysec}
\usepackage{fullpage}
\usepackage{cleveref}
\usepackage{cite}
\usepackage{url}
\usepackage{esint}
\usepackage{dsfont}
\usepackage{color}
\usepackage{overpic}

\usepackage{graphicx,epsfig}

\crefname{assumption}{assumption}{assumptions}
\crefname{thm}{theorem}{theorems}
\crefname{lem}{lemma}{lemmas}
\crefname{cor}{corollary}{corollaries}
\crefname{prop}{proposition}{propositions}
\Crefname{theorem}{Theorem}{Theorems}

\newcommand{\R}{\mathbb{R}}

\newcommand{\eps}{\varepsilon}

\newcommand{\1}{\mathds{1}}
\let\phi\varphi

\DeclareMathOperator{\supp}{supp}

\newtheorem{thm}{Theorem}[section]
\newtheorem{prop}[thm]{Proposition}
\newtheorem{lem}[thm]{Lemma}

\newtheorem{lemma}[thm]{Lemma}

\newtheorem{theorem}[thm]{Theorem}

\newcommand{\Rm}{{\mathbb R}}
\newcommand{\farc}{\frac}
\newcommand{\sgn}{{\rm sgn}}

\newcommand{\unth}{\underline{\theta}}

\date\today
\author{Emeric Bouin 
\footnote{CEREMADE - Universit\'e Paris-Dauphine, UMR CNRS 7534, Place du Mar\'echal de Lattre de Tassigny, 75775 Paris Cedex 16, France. E-mail: \texttt{bouin@ceremade.dauphine.fr}}\and  
Christopher Henderson \footnote{Ecole Normale Sup\'erieure de Lyon, UMR CNRS 5669 'UMPA', 46 all\'ee d'Italie, F-69364~Lyon~cedex~07, France. E-mail: \texttt{christopher.henderson@ens-lyon.fr}}\and
Lenya Ryzhik \footnote{Department of Mathematics, Stanford University, Stanford, CA 94305, E-mail: \texttt{ryzhik@math.stanford.edu}}}

\numberbysection

\begin{document}

\title{Super-linear spreading in local and non-local cane toads equations}
\maketitle

\begin{abstract}
In this paper, we show super-linear propagation in a nonlocal reaction-diffusion-mutation equation modeling the invasion of cane toads in Australia that has attracted attention recently from the mathematical point of view.  
The population of toads is structured by a phenotypical trait that governs the spatial diffusion.  
In this paper, we are concerned with the case when the diffusivity can take unbounded values, and we prove  
that the population spreads as $t^{3/2}$.  We also get the sharp rate of spreading in a related local model.  
\end{abstract}

\section{Introduction}

The invasion of cane toads in Australia has interesting features different from the standard spreading observed in 
most other
species. The experimental data~\cite{PhillipsBrownEtAl, Shine} show that the invasion speed has steadily increased 
during 
the eighty years since the toads were introduced in Australia.  In addition, the younger individuals at 
the edge of 
the invasion front have a significantly different morphology compared to other populations -- their legs 
tend to be on average much longer than away from the front. 
This is just one example of  a non-uniform space-trait 
distribution -- see, for instance, a study on the expansion of bush crickets in Britain~\cite{Thomas}.  Several works have addressed the  front 
invasions in ecology, where the trait is related to the
dispersal ability \cite{ArnoldDesvillettesPrevost,ChampagnatMeleard}. It has been observed
that selection of more mobile individuals can occur, even if they have no advantage 
in their reproductive rate, due to the spatial sorting 
\cite{Kokko,Ronce,PhillipsBrownEtAl,Simmons}.

In this paper, we focus on  the super-linear in time propagation  in a model of the cane toads invasion proposed in~\cite{BenichouEtAl},
based on the classical Fisher-KPP equation~\cite{Fisher,KPP}.  The population density is structured by a spatial variable, $x\in \R$, and a 
motility variable $\theta\in \Theta\stackrel{\rm def}{=}[\underline\theta, \infty)$, with a  fixed   $\underline\theta>0$. 
This population undergoes diffusion in the trait variable $\theta$, with a constant diffusion coefficient $\alpha>0$, representing 
mutation, and in the spatial variable, with the diffusion coefficient~$\theta$, representing the effect of the trait  on the spreading 
rates of the species. Thus, neglecting the competition and reproduction, the population model for the population density $u(t,x,\theta)$ would be
\begin{equation}\label{nov1002}
u_t=\theta u_{xx}+\alpha u_{\theta\theta}.
\end{equation}
In addition, each toad competes locally in space with all other individuals for resources. If the competition is local in the trait variable, 
has a saturation level $S$, and a growth rate $r$, then a Fisher-KPP type generalization of~(\ref{nov1002}) is
\begin{equation}\label{nov1004}
u_t=\theta u_{xx}+\alpha u_{\theta\theta}+ru(S-u).
\end{equation}
It is also natural to consider a non-local in trait competition (but still local in space), which leads to 
\begin{equation}\label{nov1006}
n_t=\theta n_{xx}+\alpha n_{\theta\theta}+rn(S-\rho),
\end{equation}
where
\begin{equation}\label{nov1008}
\rho(t,x)=\int_{\underline\theta}^\infty n(t,x,\theta)d\theta
\end{equation}
is the total population at the position $x$.
 
Both (\ref{nov1004}) and (\ref{nov1006}) are supplemented by the Neumann boundary condition at $\theta=\underline\theta$:
\begin{equation}\label{nov1012}
u_\theta(t,x,\underline\theta)=0,~~t>0,~x\in\Rm.
\end{equation}
The non-dimensional versions of (\ref{nov1004}) and (\ref{nov1006}) are, respectively,
\begin{equation}\label{nov1014}
u_t=\theta u_{xx}+ u_{\theta\theta}+u(1-u),~~~t>0,~x\in\Rm,~\theta\in\Theta,
\end{equation}
and 
\begin{equation}\label{nov1016}
n_t=\theta n_{xx}+ n_{\theta\theta}+n(1-\rho),~~~\rho(t,x)=\int_{\underline\theta}^\infty n(t,x,\theta)d\theta,
~~~~~t>0,~x\in\Rm,~\theta\in\Theta.
\end{equation}

In general, the speed of propagation in the Fisher-KPP type equations is determined by the linearization  around zero, 
that is, with the terms $u(1-u)$ in (\ref{nov1014}) and $n(1-\rho)$ in 
(\ref{nov1016}) replaced by $u$ and $n$, respectively.
Since the linearizations of~\eqref{nov1014} and~\eqref{nov1016} are identical, we expect both models to have the same propagation
speed.

Models involving non-local reaction terms have been the subject of intense study in recent years due to the complexity of the
dynamics -- see, for
example,~\cite{HamelRyzhik, FayeHolzer, NadinPerthameTang,NadinRossiRyzhikPerthame,AlfaroCovilleRaoul,BouinMirrahimi} 
and references therein.  
The cane toads equation has similarly attracted recent interest, mostly when the motility set $\Theta$ is a finite interval.
An Hamilton-Jacobi framework has been formally applied to the non-local model 
in~\cite{BCMetal},   and rigorously justified  in~\cite{Turanova}. 
In these works, the authors obtain the speed of propagation and the expected repartition of traits at the edge of the front by solving a spectral 
problem in the trait variable. The existence of traveling waves has been proved in
\cite{BouinCalvez}. 
The precise asymptotics of the front location for the Cauchy problem, 
up to a constant shift has been obtained in \cite{BHR_Delay} 
by using a Harnack-type inequality that allows one to compare 
the solutions of the non-local 
equation to those of a local equation, whose dynamics are well-understood~\cite{Bramson78}.   

As far as unbounded traits are concerned, a formal argument 
in~\cite{BCMetal} using a Hamilton-Jacobi framework predicted front acceleration and spreading rate of $O(t^{3/2})$.
In this paper, we give a rigorous proof of this spreading rate  both in the local and non-local models. This is  an addition to 
the growing list of ``accelerating fronts'' that have attracted some interest in recent years 
\cite{BouinCalvezNadin, HamelRoques, HendersonFast, Garnier, MeleardMirrahimi, CoulonRoquejoffre, CCR, CabreRoquejoffre}.

\subsubsection*{The local case}

Our first result concerns the local equation (\ref{nov1014}). 
%
\begin{theorem}[{\bf Acceleration in the local cane toads model}]
\label{thm:local_toads}
Let $u(t,x)$ be the solution of the local equation (\ref{nov1014}), with the boundary condition (\ref{nov1012}),
and the initial condition $u(0,x)=u_0(x)\ge 0$. Assume that  $u_0(x)$ is compactly supported, and
fix any constant $m \in (0,1)$, then
\begin{equation}\label{eq:local_asymptotics}
	\lim_{t\to\infty} \; \frac{\max\{ x \in \R : \exists \theta \in \Theta, u(t,x,\theta) = m\}}{t^{3/2}}
		= \dfrac{4 }{ 3   },
\end{equation}
The limit is uniform in $m\in[\eps,1-\eps]$, for any $\eps>0$ fixed.
\end{theorem}
  
The assumption that $u_0$ is compactly supported is made purely for convenience, one could allow 
more general   rapidly decaying or front-like initial conditions.
The proof of Theorem \ref{thm:local_toads} is in two steps.  First, we show that the Hamilton-Jacobi framework 
provides a super-solution to~\eqref{nov1014},  and gives  the upper bound of the limit in~\eqref{eq:local_asymptotics}
-- see Proposition \ref{prop:upperbound}.  Second, we prove the lower bound to the limit in~\eqref{eq:local_asymptotics} 
by constructing a sub-solution to $u$. This is done in Proposition~\ref{prop:lower_bound}.  The argument involves building 
a sub-solution to~\eqref{eq:local_asymptotics} on a traveling ball with the Dirichlet boundary conditions, whose path comes 
from the Hamilton-Jacobi framework discussed above.  These sub-solutions are arbitrarily small initially but become 
bounded uniformly away from zero on any compact subset of the traveling ball for large time.  The analysis is 
complicated by the fact that the diffusivity is unbounded  
in the $\theta$ direction.   It is crucial to match
the upper bound that we use the optimal paths coming from the Hamilton-Jacobi framework. 


\subsubsection*{The non-local case}

Our second main result shows that the full non-local model~\eqref{nov1016} exhibits a similar front acceleration.
\begin{theorem}[{\bf Acceleration in the non-local cane toads model}]
\label{thm:nonlocal_toads}
Let $n(t,x)$ be the solution of the non-local cane toads equation (\ref{nov1016}), with the Neumann boundary condition (\ref{nov1012}),
and the initial condition $n(0,x)=n_0(x)\ge 0$. Assume that  $n_0(x)$ is compactly supported, and
fix any $\eps > 0$.  
There exists a positive constant $c_\eps$, depending only on $\eps$, such that
\begin{equation}\label{eq:nonlocal_lower}
	\frac{8}{3\sqrt{3\sqrt{3}}} (1-\eps) \leq \limsup_{t\to\infty} \frac{\max\left\{x \in \R: \rho(t,x) \geq c_\eps \right\}}{t^{3/2}}.
\end{equation}
In addition, if $m$ is any constant in $(0,1)$, then we have that
\begin{equation}\label{eq:nonlocal_upper}
\limsup_{t\to\infty} \frac{ \max\{ x \in \R : \rho(t,x) \geq m\}}{t^{ 3/2}} \leq  \dfrac{4 }{ 3 }.
\end{equation}
\end{theorem}
In contrast to the sharp bound provided by Theorem~\ref{thm:local_toads}, Theorem~\ref{thm:nonlocal_toads}  
does not have matching lower and upper bounds.  This is due to the lack of a comparison principle for~\eqref{nov1016}.  
The proof of the lower bound involves arguing by contradiction in order to   compare~\eqref{nov1016} 
to a solution of the local, linear equation defined on a moving ball with the Dirichlet boundary conditions. 
 This causes the $\limsup$ to appear in~\eqref{eq:nonlocal_lower}, 
as opposed to the more desirable $\liminf$.  In addition, because the optimal Hamilton-Jacobi trajectories
trajectories are initially almost purely in the $\theta$ direction, we can not use them to move the ``Dirichlet ball", and
are forced to use non-optimal trajectories, leading to the sub-optimal 
constant $8/(3\sqrt{3\sqrt3})$ in (\ref{eq:nonlocal_lower}).  We comment on this in 
Section~\ref{ss:nonlocal_lower}.  We believe
that the sharp result would have the   lower bound matching 
the upper bound~\eqref{eq:nonlocal_upper}.  The proof of the upper bound is given by 
Proposition~\ref{prop:upperbound} below, as in the local case, while the proof of the lower bound and an explicit bound on $c_\eps$
are given by Proposition~\ref{prop:nonlocal_lowerbound}.   
We also note that in general non-local Fisher-KPP type equations the stability of the steady state~$u\equiv 1$ may 
fail~\cite{BNPR,GenieysVolpertAuger,Gourley}.  Thus, it is not surprising that we are restricted to working with 
the level sets of certain heights $c_\eps$ in~\eqref{eq:nonlocal_lower} -- this
mirrors the propagation results in~\cite{HamelRyzhik}.  


We should mention the concurrent work by 
Berestycki, Mouhot, and Raoul~\cite{BerestyckiMouhotRaoul}, who  
use a mix of probabilistic and analytic methods to 
prove the same sharp result in the local model~\eqref{nov1014}.  
In addition, they 
prove the sharp asymptotics in a 
non-local model where~$\rho$ is replaced by a windowed non-local.
Ecologically, the windowed term models a situation where individuals compete for 
resources only with individuals of a similar trait.
The techniques of~\cite{BerestyckiMouhotRaoul} 
do not seem to apply to the full non-local model that we address here.  
%


The rest of the paper is organized as follows. In Section \ref{sec:HJ} we recall  some facts from \cite{BCMetal} on 
the Hamilton-Jacobi framework. Then we derive in Section \ref{sec:upper} the upper bound 
common to Theorems~\ref{thm:local_toads} and~\ref{thm:nonlocal_toads} using an explicit super-solution 
that arises from the work in Section \ref{sec:HJ}.  The lower bound is then proved in Section~\ref{sec:lower}.  There, 
we first derive a general propagation result on the linearized equation by using the optimal paths from Section~\ref{sec:HJ}, 
re-scaling the equation 
appropriately, and using precise spectral estimates.  We then use this result to obtain the lower bounds for the local and non-local models.
Section~\ref{sec:trajectories} contains the proofs of some auxiliary results.

{\bf Acknowledgments.}  The authors wish to thank Vincent Calvez and Sepideh Mirrahimi for fruitful discussions and earlier computations on this problem. LR was supported by NSF grant DMS-1311903. EB was supported by ``INRIA Programme Explorateur" and is very grateful to Stanford University for its sunny hospitality during the second semester of the academic year 2014-2015. CH acknowledges the support of the Ecole Normale Sup\'erieure de Lyon for a one week visit in April 2015.  Part of this work was performed within the framework of the LABEX MILYON (ANR-10-LABX-0070) of Universit\'e de Lyon, within the program ``Investissements d'Avenir" (ANR-11-IDEX-0007) operated by the French National Research Agency (ANR).

\section{The Hamilton-Jacobi solutions}\label{sec:HJ}

In this section, we recall how a suitable scaling limit of the cane toads equation can formally give the acceleration rate~\cite{BCMetal}. 
The analysis of this section will be used in the rest of the paper 
to construct ``good" sub- and super-solutions to  the local and local equations \eqref{nov1014} and \eqref{nov1016}. We will focus on the linearized
equation
\begin{equation}\label{nov1302}
u_t=\theta u_{xx}+u_{\theta\theta}+u.
\end{equation}
Writing 
\[
u(t,x)=e^t v(t,x),
\]
we reduce it to
\begin{equation}\label{nov1304}
v_t=\theta v_{xx}+v_{\theta\theta}.
\end{equation}
One obvious solution to this equation is
\[
v(t,\theta)=\frac{1}{\sqrt{4\pi t}}e^{-\theta^2/(4t)},
\]
which also provides a spatially uniform super-solution to (\ref{nov1302}):
\begin{equation}\label{nov1306}
\bar v(t,\theta)=e^{t-\theta^2/(4t)}.
\end{equation}
However, the function $\bar v(t,\theta)$ has no decay in $x$, and is not 
useful in the spatial regions where we expect the solution of the full 
cane toads equation to be small. 

In order to construct another super-solution to (\ref{nov1304}), with decay in
space,
we  rescale  \eqref{nov1304} setting 
\[
v^\eps(t,x,\theta)=v \left(  \frac{t}{\eps}, \frac{x}{\eps^{3/2}}, \frac{\theta}{\eps} \right).
\]
The function $u^\eps(t,x,\theta)$ satisfies
\begin{equation*} \label{eq:scaling 3ter}
\eps \partial_t u^\eps  = \eps ^2\theta \partial^2_{xx} u^\eps + \eps^2 \partial^2_{\theta\theta} u^\eps.
\end{equation*}
We perform the \textit{Hopf-Cole transformation} 
\begin{equation*}
u^\eps=\exp\{-\varphi^\eps/\eps\},
\end{equation*} 
so that
\begin{equation}\label{nov1316}
\partial_t \varphi^\eps + \theta|\varphi^\eps_x|^2 
+ |\varphi^\eps_\theta|^2=\eps\theta\varphi^\eps_{xx}
+\eps\varphi^\eps_{\theta\theta},
\end{equation}
and obtain, in the formal limit as $\eps\to 0$, the Hamilton-Jacobi equation
\begin{equation}\label{eq:HJ}
\partial_t \psi + \theta|\psi_x|^2 + |\psi_\theta|^2=0.
\end{equation}
We will  use the solutions of this Hamilton-Jacobi
equation to construct  
sub- and super-solutions to the original problem. 
The Hopf-Cole transformation is an effective tool in the analysis
of front propagation for reaction-diffusion equations -- see, for instance,
\cite{LE.PS:89,WF.PS:86,GB.LE.PS:90}, including 
parabolic integro-differential equations modeling populations structured by a phenotypical trait (see \cite{OD.PJ.SM.BP:05,GB.BP:08,BouinMirrahimi}).

\subsubsection*{A ``heat kernel'' solution of the Hamilton-Jacobi equation}

Let us consider the Hamilton-Jacobi equation (\ref{eq:HJ}) 
with the initial condition
\begin{equation}\label{nov1022}
\psi(t=0,x,\theta) = 
\begin{cases} 0 & \text{if $(x,\theta) = (0,0)$} \\
+ \infty & \text{if $(x,\theta) \neq  (0,0)$}.
\end{cases}
\end{equation}
The reason for this choice of the initial data is clear:
for the standard heat equation, the Hamilton-Jacobi equation would be
\begin{equation}\label{nov1314}
\psi_t+|\psi_x|^2=0,
\end{equation}
and the solution is simply
\begin{equation}\label{nov1310}
\psi(t,x)=\farc{x^2}{4t},
\end{equation}
leading to the standard heat kernel.
The Hamiltonian for~\eqref{eq:HJ} is 
\begin{equation}\label{nov1312}
H((x,\theta),(p_x,p_\theta)) = \theta |p_x|^2 + |p_\theta|^2,
\end{equation}
and the corresponding Lagrangian is
\begin{equation}\label{dec402}
L((x,\theta),(v_x,v_\theta)) = \frac{v_x^2}{4\theta} + \frac{v_\theta^2}{4}.
\end{equation}
Using the Lax-Oleinik formula to solve~\eqref{eq:HJ}, we get
\begin{equation}
\psi  (t, x , \theta)
= \inf_{w\in C^1} \left \lbrace \int_{0}^{t} L(w(s), \dot w(s)) ds \; \Big\vert \; w(0) = ( x , \theta ) , w(t) = (0,0) \right \rbrace.
\end{equation}
The  optimal trajectory given by the Hamiltonian flow is the solution of
\[
\frac{dX(s)}{ds}=2P_x(s)\theta(s),~~~\frac{d\theta(s)}{ds}=2P_\theta(s),~~~\frac{dP_x(s)}{ds}=0,~~~\farc{dP_\theta(s)}{ds}=-P_x(s)^2.
\]
hence
\[
\frac{d}{ds}\Big(\farc{1}{\theta(s)}\farc{dX(s)}{ds}\Big)=0,~~
\farc{d^2\theta(s)}{ds^2}=-\farc{1}{2}\Big(\farc{1}{\theta(s)}\farc{dX(s)}{ds}\Big)^2.
\]

Thus, there is a constant $C$ such that
\begin{equation}\label{eq:odes}
\dot X (s) = C \theta(s), \qquad \ddot \theta(s) = - \frac{C^2}{2},~~X(0)=x,~\theta(0)=\theta,~~X(t)=0,~\theta(t)=0.
\end{equation} 
Plugging this into the expression for $\psi$ gives
\begin{equation}\label{eq:psi_computation}
\begin{split}
\psi(t, x, \theta)
&= \int_{0}^{t} L(w(s), \dot w(s)) ds
=  \int_{0}^{t} \frac14 \left( \frac{\dot X(s)^2}{\theta(s)} + \dot \theta(s)^2 \right) ds \\
& = \frac{C}{4} \displaystyle\int_{0}^{t} \dot X(s) ds + 
\farc{t\dot\theta(t)^2}{4} -\farc{1}{2}\int_0^t s\dot\theta(s)\ddot\theta (s)ds
 =  - \frac{Cx}{4} +  \frac{t  \dot \theta(t) ^2}{4}  	+ 
 \farc{C^2}{4} \displaystyle \int_{0}^{t} s \dot \theta(s) ds\\
& = - \frac{Cx}{4} +  \frac{t \dot \theta(t)^2}{4 } 
-  \farc{C^2}{4} \displaystyle \int_{0}^{t} \theta(s) ds	   
= - \frac{Cx}{4} + \frac{t \dot \theta(t)^2}{4} +  \farc{C x}{4} 
= \frac{t\dot \theta(t)^2}{4}.
\end{split}
\end{equation}
We now compute $\dot \theta(t)$. From~\eqref{eq:odes}, 
we find that
\begin{equation}\label{eq:theta_trajectory}
	\theta(s) = - \frac{C^2}{4} s^2 + \left( \frac{C^2t}{4} - \frac{\theta}{t} \right) s + \theta,
\end{equation}
which implies that
\begin{equation}\label{eq:theta_dot}
	\dot \theta(t) = - \frac{\theta}{t} - \frac{C^2 }{4} t = - \frac{1}{t} \left( \theta + \frac{C^2 t^2}{4}\right).
\end{equation}
To obtain a closed form for $\psi$, we need to find $C$.  
We use~\eqref{eq:odes} to obtain
\[
- x= \int_0^t \dot X(s) ds= C \int_0^t \theta(s) ds
= C\left( -\frac{C^2 t^3}{12} + 
\left(\frac{C^2t}{4} - \frac{\theta}{t}\right)\frac{t^2}{2} + \theta t\right)
= \frac{(Ct)^3}{24} + \frac{\theta (Ct)}{2}.
\]
It follows that $Z = Ct/2$ is the unique real solution of the cubic equation
\begin{equation}\label{W}
Z^3 + 3\theta Z +  3x  = 0.
\end{equation}

Combining~\eqref{eq:psi_computation},~\eqref{eq:theta_dot}, and~\eqref{W}, 
we obtain an explicit formula for $\psi(t,x)$:
\begin{equation}\label{eq:psi}
\psi ( t , x,\theta) = \frac{1}{4 t} 
\left( \theta + {Z(x,\theta)^2}  \right)^2,
\end{equation}
the analog of (\ref{nov1310}) for our problem.

\subsubsection*{Super-solutions with the diffusion}

The function $\psi(t,x)$ was obtained neglecting the right side in (\ref{nov1316}).
It turns out, that, when the diffusion is taken into account, it still
leads to a super-solution to the linearized cane toads equation
\begin{equation}\label{nov1130}
u_t-\theta u_{xx}-u_{\theta\theta}-u\ge 0,
\end{equation}
of the form
\begin{equation}\label{nov1132}
\bar u(t,x,\theta)=\exp\{t-\psi(t,x,\theta)\},
\end{equation}
or, more explicitly showing the difference with (\ref{nov1306}):
\begin{equation}\label{nov1320}
\bar u(t,x,\theta)=\exp\left\{t-\frac{1}{4t}(\theta+Z^2(x,\theta))^2\right\}.
\end{equation}
This function satisfies
\begin{equation}\label{nov1134}
\bar u_{t}-\theta \bar u_{xx}-\bar u_{\theta\theta}-\bar u =
\bar u\big(-{\psi}_{t}+\theta\psi_{xx}-
\theta|{\psi}_x|^2
+\psi_{\theta\theta}-|{\psi}_{\theta}|^2\big)=\bar u\big(\theta\psi_{xx}+
\psi_{\theta\theta}\big).
\end{equation}
Observe that  
\begin{equation}\label{nov1136}
\begin{split}
&4t\big( \theta\psi_{xx}+\psi_{\theta\theta}\big)  
 = 4\theta(( \theta + {Z^2})Z Z_x)_{x} + 
2( ( \theta + {Z^2})(1 +{2} Z Z_\theta))_{\theta} \\
& = 4( \theta + {Z^2})( \theta (Z Z_x)_{x} +  (Z Z_\theta)_{\theta}) + 
8{ \theta }( Z Z_x)^2 + 2(1 + {2} Z Z_\theta)^2
\ge  4( \theta + {Z^2}) 
Z \left( \theta  Z_{xx} + Z_{\theta \theta}\right) .
\end{split}
\end{equation}
We claim that, somewhat miraculously, 
\begin{equation}\label{nov1140}
\theta Z_{xx} + Z_{\theta\theta} = 0.
\end{equation}
Indeed, we have
\begin{equation}\label{nov1326}
Z_\theta = -\frac{Z}{ Z^2 +\theta},  ~~ Z_x = -\frac{1}{ Z^2 + \theta},
\end{equation}
so that
\[
Z_{\theta\theta}
= -\frac{Z_\theta ( Z^2 + \theta) - Z (2  Z Z_\theta + 1 )}{( Z^2 +  \theta)^2}
= \frac{Z ( Z^2 + \theta)+\left(\theta - Z^2 \right) Z }
{( Z^2 + \theta)^3}
= \frac{2Z \theta}{( Z^2 + \theta)^3},
\]
and
\[
Z_{xx} = -\frac{2 Z}{( Z^2 + \theta)^3}.
\]
Thus, (\ref{nov1140}) holds and, as a consequence, (\ref{nov1130}) follows.

\subsubsection*{A level set of the super-solution and the optimal trajectories}

The level set $\{\bar u(t,x,\theta)=1\}$ is given by
\begin{equation}\label{nullset}
	\theta + {Z(x,\theta)^2}= 2t.
\end{equation}
Multiplying this equation by $Z$, and using \eqref{W}, we get 
\begin{equation}
3 \theta Z + 3 x
		= - \left( 2t - \theta \right) Z,
\end{equation}
or
\begin{equation}\label{eq:Z}
Z(x,\theta) = -\frac{3x}{2(\theta + t)}. 
\end{equation}
Inserting this back into \eqref{nullset}, gives the equation for the level
set $\{\bar u(t,x,\theta)=1\}$:
\begin{equation}\label{nullset2}
x^2 = \frac{4}{9} \left( 2t - \theta \right) \left( \theta + t \right)^2.
\end{equation}
In order to compute the rightmost point of this level set at a given
time $t>0$, we
differentiate~(\ref{nullset2}) in~$\theta$.
The critical points are determined by
\[
0=2x\frac{\partial x}{\partial\theta}
= \frac{4}{9} \left[ -(\theta+t)^2 + 2(2t -\theta) (\theta+t)\right]
= \frac{4}{3} (t+\theta)(t-\theta),
\]
and the maximum
\begin{equation}\label{nov1026}
x_{\rm edge}(t)=  \dfrac{4 }{ 3   } t^{3/2}\,
\end{equation}
is attained at $\theta_{\rm edge}(t)=t$.

With the maximal end-points in hand, we now compute the Lagrangian trajectories
$X(s)$,~$\theta(s)$ which travel to the 
far edge $(x_{\rm edge}(t),\theta_{\rm edge}(t))$.  
Using the general form~\eqref{eq:theta_trajectory} with the 
choice $C = 2Z/t$ for $\theta(s)$, we obtain
\begin{equation*}
\theta(s)
= -{Z(x_{\rm edge}(t),\theta_{\rm edge}(t))^2}\left(\frac{s}{t}\right)^2 + 
({Z(x_{\rm edge}(t),\theta_{\rm edge}(t))^2} - \theta_{\rm edge})\frac{s}{t} 
+ \theta_{\rm edge}(t).
\end{equation*}
Then using the expressions $x_{\rm edge}(t) = 4 t^{3/2}/3$ and $\theta_{\rm edge}(t) = t$ along with the expression~\eqref{eq:Z} for~$Z$, we obtain
\begin{equation}\label{nov1028}
\theta(s)
= - \frac{9x_{\rm edge}(t)^2}{16t^2} 
\left(\frac{s}{t}\right)^2 
+ \Big(\frac{9x_{\rm edge}(t)^2}{16t^2} - t\Big)\frac{s}{t} + t 
= t\Big(1 - \left(\frac{s}{t}\right)^2\Big).
\end{equation}
We similarly deduce the trajectory for $X(s)$.  Indeed, we have from~\eqref{eq:odes} and the definition of $Z$ that
\[
X(s) = x_{\rm edge}(t) + \frac{2Z}{t} \int_0^s \theta(r) dr.
\]
Using expression (\ref{nov1026}) for
$x_{\rm edge}(t)$  and (\ref{nov1028}), we get 
\begin{eqnarray*}
&&X(s)= \frac{4}{3}t^{3/2} - 
\frac{1}{t}\left(\frac{3x_{\rm edge}}{\theta_{\rm edge} + t}\right) 
\int_0^s t \left( 1 - \left(\frac{r}{t}\right)^2\right) dr
= \frac{4}{3}t^{3/2} - 2 t^{1/2} \left( s - \frac{s^3}{3t^2}\right)\\
&&~~~~~~= \frac{4}{3} t^{3/2} \left( 1 - \frac{3s}{2t} 
\left(1 - \frac{s^2}{3t^2}\right)\right).
\end{eqnarray*}
To obtain the forward trajectories, we reverse time, that is, 
change variables from $s \mapsto t -s$, and with a slight abuse of
notation, write
\begin{equation}\label{dec706}
X(s) = \left(\frac{3}{2} - \frac{s}{2t}\right) 
\left(\frac{s}{t}\right)^2 \frac{4}{3} t^{3/2},
~~~\text{ and }~~~\theta(s) = s \left(2 - \frac{s}{t}\right).
\end{equation}

\subsubsection*{The minimum of $\psi(t,x,\theta)$}

In the sequel, we also need the minimum $\theta_*(t,x)$
of $\psi(t,x,\theta)$ in $\theta\in\Theta$,
for $t$ and $x$ fixed.
We differentiate expression~(\ref{eq:psi}) for $\psi(t,x,\theta)$  
with respect to $\theta$:
\begin{equation}\label{nov1144}
\psi_\theta = \frac{1}{2 t} \left( \theta + Z^2 \right) 
\left( 1 + {2} Z Z_\theta \right).
\end{equation}
Hence, the critical points satisfy
\begin{equation}\label{nov1102}
Z Z_\theta = - \farc 12.
\end{equation}

Using expression (\ref{nov1326}) for $Z_\theta$   leads to
(note that $Z$ has the opposite sign of $x$)
\begin{equation*}
Z=-\sqrt{\theta_*(x)}\sgn(x).
\end{equation*}
We insert this value into~\eqref{W}, and find
\begin{equation}\label{nov1106}
 \theta_*(x) = \left( \frac{3}{4} |x| \right)^{2/3},
\end{equation}
and
\begin{equation}\label{nov1107}
\psi(t,x,\theta_*(x)) = \frac{1}{t} \left( \frac{3}{4} |x| \right)^{4/3}.
\end{equation}
Note that the internal minimum exists only if $|x|$ is sufficiently large:
\begin{equation}\label{nov1154}
|x|>r_c \stackrel{\rm def}{=} \farc{4}{3}\unth^{3/2},
\end{equation}
so that
$\theta_*(x)\ge\underline\theta$, otherwise the minimum of $\psi(t,x,\theta)$ 
is attained at $\theta=\underline\theta$.

\begin{figure}[h]
\begin{center}
\includegraphics[width = 0.65\linewidth]{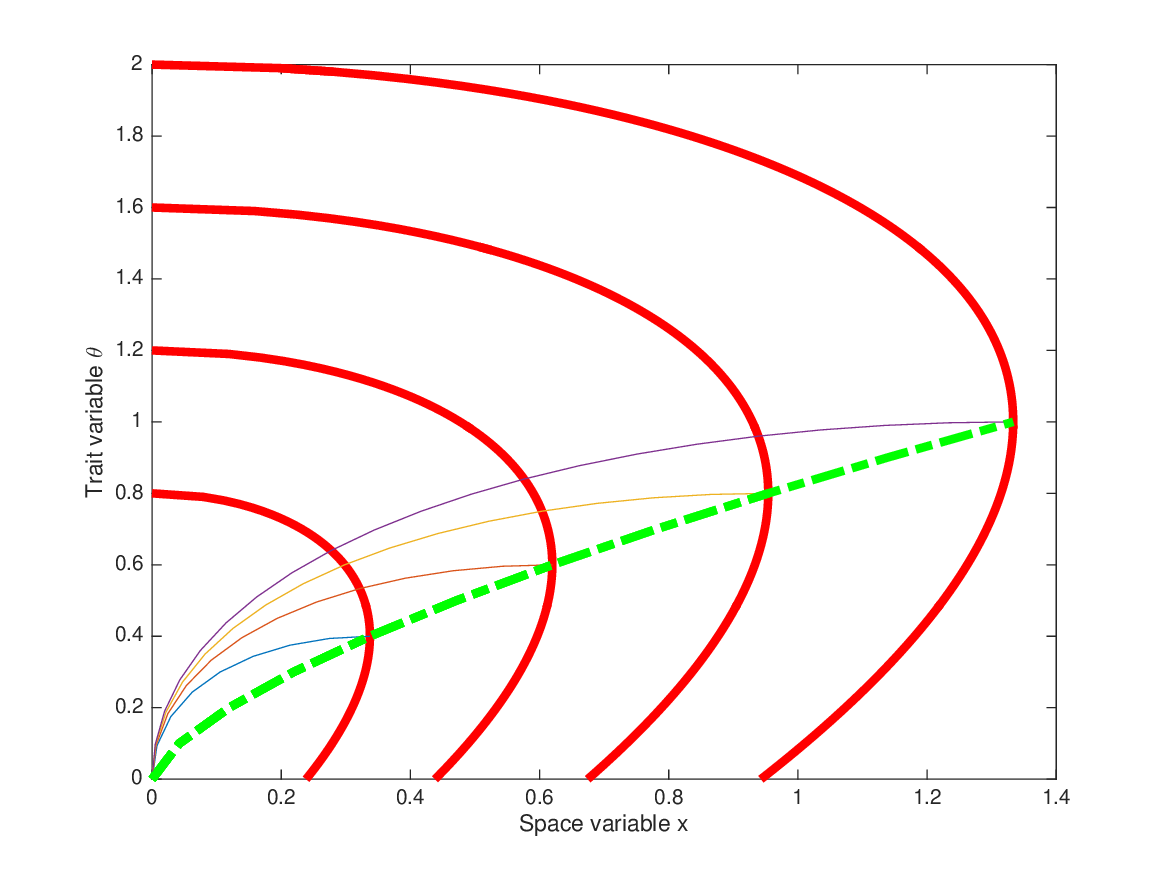}
\end{center}
\caption{The level set $\{\bar u(t,x,\theta)=1\}$ in the phase space $(x, \theta)$ for different values of time 
and the optimal trajectory, see also \cite{BCMetal}. We also plot the trajectories that lead to the edge of the front for various values of time. These latter trajectories are crucial for the paper.}\label{f:level_sets}
\end{figure}

\section{An upper bound }\label{sec:upper}

In this section,     
we construct an explicit super-solution for the local and nonlocal versions of the 
cane toads equation. This provides an upper bound on the spreading rate.

\subsection*{A super-solution}

Ideally, we would take as a super-solution the function 
\[
u(t,x,\theta)=\bar u(t+a,x,\theta),
\]
with $\bar u$ as in (\ref{nov1132}) with some 
suitably chosen $a>0$. 
There is an obstacle: a super-solution, in addition to
(\ref{nov1130}), should satisfy the boundary condition
\begin{equation}\label{nov1148}
u_\theta(t,x,\underline\theta)\le 0.
\end{equation}
For a function of the form (\ref{nov1132}), this condition is equivalent to
\begin{equation}\label{nov1150}
\psi_\theta(t,x,\underline\theta)\ge 0.
\end{equation}
The function $\psi(t,x,\theta)$ has a single minimum, 
either at $\theta=\unth$, if~$|x|<r_{c}$, or at $\theta=\theta_*(x)$ if~$|x|>r_c$, where we recall $r_c$ from~\eqref{nov1154}.
Hence, (\ref{nov1150}) can not hold for $|x|>r_c$, and 
we need to modify~$u(t,x,\theta)$ to turn it
into a true super-solution.
 
To this end, we recall the other family of
super-solutions,~$\bar v(t,\theta)$ given
by (\ref{nov1306}):
\begin{equation}\label{nov1306bis}
\bar v(t,\theta)=e^{t-\theta^2/(4t)},
\end{equation}
that do satisfy the boundary condition (\ref{nov1148}). As, on the other
hand, $\bar v(t,\theta)$ has no decay in $x$, we will only use it for $x<0$.
Let us define the set
\begin{equation}\label{nov1161}
\Omega = \left\lbrace (x,\theta):~x\ge r_c,
\hbox{ and } 
\underline{\theta}\leq \theta \leq \theta_*(x) \right\rbrace.
\end{equation}
We define our super-solution first on $\{x \leq 0\}$ and on $\{x\geq 0\}\cap \Omega^c$, but we extend it to all of $\R\times \Omega$ below.  First, we define
\begin{equation}\label{nov1122}
\tilde u(t,x,\theta)=\begin{cases} C(a)\bar v(t+a,\theta),~~~~~&x\le 0,\\
C(a)\bar u(t+a,x,\theta), &x\ge0,~~(x,\theta)\in\Omega^c,
\end{cases}
\end{equation}
with the constant $C(a)$ to be chosen later.
Note that we have
\begin{equation}\label{nov1330}
\bar v(t+a,\theta)\ge\bar u(t+a,x,\theta),
\end{equation}
with the equality holding only at $x=0$, where we recall that $Z(t,x=0,\theta)=0$.
Hence, the function~$\tilde u(t,x,\theta)$ is $C^1$ in $\Omega^c$.
 
It is a super-solution in the sense that
\begin{eqnarray}\label{nov1162}
&&\tilde u_t-\theta\tilde u_{xx}-\tilde u_{\theta\theta}-
\tilde u\ge 0\hbox{ in $\Omega^c$},\\
&&\bar u_\theta(t,x,\unth)\le 0\hbox{ for all $x\le r_c$.}\nonumber
\end{eqnarray} 

In order to extend $\tilde u$ into $\Omega$ as a $C^1$-supersolution we simply set
\begin{equation}\label{nov1164}
\tilde u(t,x,\theta)=C(a)
\bar u(t,x,\theta_*(x)),~~\hbox{ for $(x,\theta)\in\Omega$,}
\end{equation}
or, more explicitly
\begin{equation}\label{nov1166}
\tilde u(t,x,\theta)= C(a)
\exp\Big\{t+a -\farc{1}{t+a}\Big(\frac{3x}{4}\Big)^{4/3}\Big\},~~
\hbox{ for $(x,\theta)\in\Omega$}.
\end{equation}
\begin{figure}[t]
\centering
	\begin{overpic}[width = 1. \linewidth]{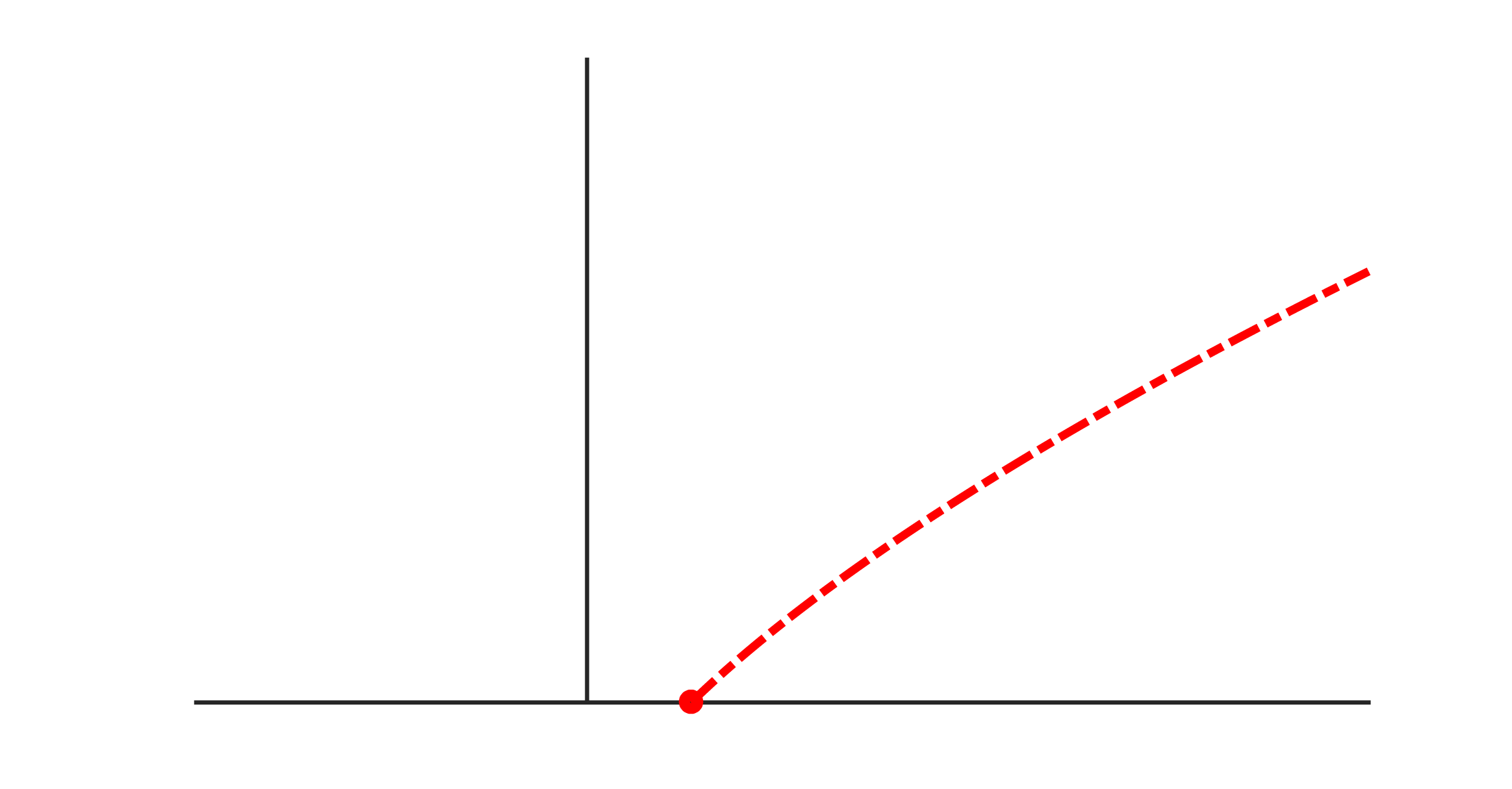}
		\put(10,30){$C(a) \exp\left\{t + a - \frac{\theta^2}{4(t+a)}\right\}$}
		\put(45,35){$C(a) \exp\left\{t+a - \frac{( \theta + Z(x,\theta)^2)^2}{4(t+a)}\right\}$}
		\put(60,12){$C(a) \exp\left\{t + a - \frac{1}{t+a}\left(\frac{3x}{4}\right)^{4/3}\right\}$}
		\put(84,36){\color[rgb]{1.0,0,0}$\theta = \theta_*(x)$}
		\put(90,3){$x$}
		\put(36,47){$\theta$}
	\end{overpic}
\caption{Definition of the super-solution in the three domains $\{x < 0\}$, $\Omega$ and $\Omega^c\cap\{x \geq 0\}$.  The red line gives the curve $(x,\theta_*(x))$.}\end{figure}
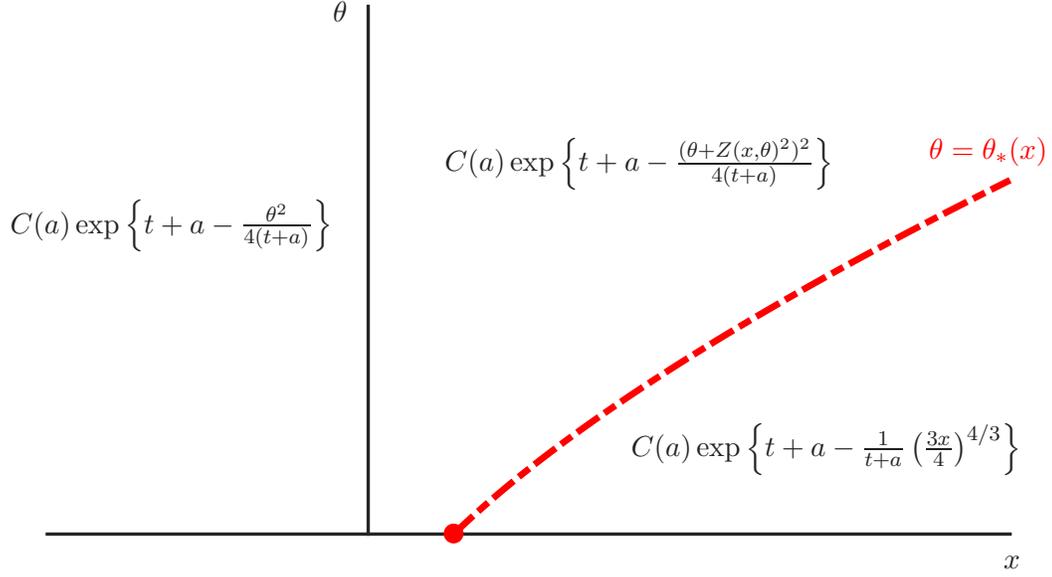

We need to verify that the extended function is $C^1$ and that it satisfies (\ref{nov1162}). Before we verify this, we point out that in order to apply the maximum principle, we only require that $\tilde u$ is $C^1$ and is $C^2$ except on a smooth 1D set.  As such, we need not worry that $\tilde u$ is likely not $C^2$ on $\partial \Omega$.  The boundary condition
in (\ref{nov1162}) is automatic since $\tilde u$ does not depend on 
the variable $\theta$ in $\Omega$.  
As $\theta_*(x)$ is the minimum of $\psi(t,x,\theta)$ in $\theta$, 
we have $\tilde u_\theta=0$ on both sides of
the curve 
\[
\Gamma=\partial\Omega=\{(x,\theta_*(x)),~x\ge r_c\}.
\]
It is easy to see that the $x$-derivatives of $\tilde u$ 
match across $\Gamma$ for the same reason, and the extended function is
$C^1$. We now compute in $\Omega$:
\begin{equation}\label{nov1168}
\begin{split}
\tilde u_t-\theta\tilde u_{xx}-\tilde u_{\theta\theta}-\tilde u
&=\tilde u\Big[
\farc{1}{(t+a)^2}\Big(\frac{3x}{4}\Big)^{4/3}+1-
\farc{\theta}{(t+a)^2}\Big(\farc{3x}{4}\Big)^{2/3}
+\farc{\theta}{t+a}\Big(\frac{3}{2x}\Big)^{2/3}-1\Big] \\
&\ge \farc{\tilde u}{(t+a)^2}\Big(\farc{3x}{4}\Big)^{2/3}\Big[\Big(\farc{3x}{4}\Big)^{2/3}-\theta\Big]=
\farc{\tilde u}{(t+a)^2}\Big(\farc{3x}{4}\Big)^{2/3}(\theta_*(x)-\theta)
\ge 0,\nonumber
\end{split}
\end{equation}

as $\theta\le\theta_*(x)$ in $\Omega$. Thus, 
the function $\tilde u(t,x,\theta)$ is a super-solution 
in the sense of (\ref{nov1162})
in the whole domain $(x,\theta)\in\Rm\times\Theta$:
\begin{eqnarray}\label{nov1202}
&&\tilde u_t-\theta\tilde u_{xx}-
\tilde u_{\theta\theta}-\tilde u\ge 0
\hbox{ for all $x\in\Rm$ and $\theta\in\Theta$},\\
&&\tilde u_\theta(t,x,\unth)\ge 0\hbox{ for all $x\in\Rm$.}\nonumber
\end{eqnarray} 

\subsection*{An upper bound}

We now use the above super-solution to give an upper bound for
the speed of propagation for the local and non-local cane toads equations.
\begin{prop}\label{prop:upperbound}
Let $u_0 \in L^\infty(\R \times \Theta)$ be a non-zero, 
non-negative function such that 
\[
u_0 \leq \1_{[-\infty, x_r]\times[\underline\theta, \theta_u]}
\]
for some $x_r \in \R$, $\theta_u > \underline\theta$. Fix any $m>0$.  Let $u$ be a solution of either 
the local cane toads equation~\eqref{nov1014} or the non-local 
cane toads equation \eqref{nov1016}. Then the following upper bounds hold
\begin{equation*}
\begin{split}
	&\lim_{t\to\infty} \frac{ \max\{ x \in \R : \exists \theta \in \Theta, u(t,x,\theta) \geq m\}}{t^{3/2}} \leq  \dfrac{4 }{ 3 }, ~~\text{ and}\\
	&	\lim_{t\to\infty} \frac{ \max\{ x \in \R :\rho(t,x,\theta) \geq m\}}{t^{3/2}} \leq  \dfrac{4 }{ 3 }.
\end{split}
\end{equation*}
\end{prop}
\begin{proof}[{\bf Proof}]
Whether $u$ solves~\eqref{nov1014} or~\eqref{nov1016},  it satisfies 
\begin{eqnarray*}
&&u_t - \theta u_{xx} - u_{\theta \theta} - u \leq 0, \qquad (t,x,\theta) \in \R^+ \times \R \times \Theta, \\
&&u_\theta(t,x,0)=0, \\
&&u(t=0,x,\theta) = u_0(x,\theta).
\end{eqnarray*}
As a consequence, $u$ is a sub-solution of the linearized cane toads equation.  On 
the other hand, the function $\tilde u(t,x,\theta)$ defined
by (\ref{nov1122}) and (\ref{nov1166}) is a super-solution, in the sense that 
(\ref{nov1202}) holds.
Let us choose a constant $C(a)$ large enough so that, for all $(x,\theta)\in \R\times \Theta$,
\[
	\tilde u (0,x,\theta) \geq u_0(x,\theta).
\]
We deduce from the parabolic comparison principle  that 
\begin{equation}\label{nov1204}
u(t,x,\theta) \leq \tilde{u}(t,x,\theta),
\end{equation} 
for all $t\ge 0$, $x\in\Rm$, and $\theta\in\Theta$. 

We now use the explicit expression for $\tilde u$ 
to obtain an upper bound on the location of the level sets of $u$.
To this end, fix any $m \in (0,1)$. As $\theta_*(x)$ is 
the minimum of the function $\psi(t,x,\theta)$, the rightmost point $\tilde x_m(t)$
of the 
level set $\{\tilde u(t,x,\theta)=m\}$
is where it intersects the curve~$\theta=\theta_*(x)$:
\begin{equation}\label{nov1340}
\tilde x_m(t) = \frac43 (t+a)^{3/2} 
\Big( 1 - \farc{1}{t+a}\log\Big( \frac{m}{C(a)} \Big)  \Big)^{3/4}. 
\end{equation}  

Using (\ref{nov1204}) and passing to the limit $t\to+\infty$, we obtain
from (\ref{nov1340}):
\begin{equation*}
\lim_{t\to\infty} \frac{\max\{ x \in \R : \exists \,\theta \in \Theta, u(t,x,\theta) \geq m\}}{t^{ \frac32} } \leq  \dfrac{4 }{ 3 }  .
\end{equation*}
This gives us an upper bound on the level sets of $u$, 
but we also need an upper bound on $\rho$.  As follows from (\ref{nov1204}),
it suffices to bound
\[
\tilde \rho(t,x)=\int_{\unth}^\infty \tilde u(t,x,\theta)d\theta.
\]
Let us use $\tilde m = m / 10(t+a)$ in (\ref{nov1340}), so that
\[
	\tilde u(t,x_{m,t},\theta) \leq \frac{m}{10(t+a)},~~\hbox{ for all $\theta\ge\unth$,}
\]
where we define
$x_{m,t} = \tilde x_{m/10(t+a)}(t)$.  We recall that
\begin{equation}\label{nov1345}
\theta_* (x_{m,t})= \left(\frac{3x_{m,t}}{4}\right)^{2/3}.
\end{equation}
Note that, for $t$ large enough we have
\begin{equation}\label{nov1346}
\farc{5}{6}(t+a)\le \theta_* (x_{m,t})\le \farc{7}{6}(t+a).
\end{equation}
Then we have 
\[
\int_0^{5\theta_*(x)}\tilde u(t,x_{m,t},\theta)d\theta
\le 5\theta_*(x)\frac{m}{10(t+a)}\le\frac{5m}{6}.
\]
We now consider the integral from $5\theta_*(x_{m,t})$ 
to $\infty$.  
We write, using (\ref{nov1330}) and (\ref{nov1346}):
\[
\begin{split}
\int_{5\theta_*(x_{m,t})}^\infty \tilde u(t, x_{m,t}, \theta)d\theta
&\leq C(a) \int_{5\theta_*(x_{m,t})}^\infty 
\exp\Big( t+a-\frac{\theta^2}{4 (t + a)}\Big) d\theta\\
&\leq C'(a) \frac{t+a}{\theta_*(x_{m,t})} 
\exp\Big\{t+a- \frac{25\theta_*^2(x_{m,t})}{4(t+a)}\Big\} \le
C'(a)e^{-(t+a)/10}.
\end{split}
\]
If $t$ is sufficiently large, depending only on $a$ and $m$, the last two estimates give us
\[
\limsup_{t\to\infty} \int_{\underline\theta}^\infty \overline u(t,x_{m,t},\theta)d
\theta\leq m.
\]
Since $\overline u$ is monotonic in the spatial variable $x$, 
we have, for any $x \geq 0$
\[
\limsup_{t\to\infty} \rho(t,x + x_{m,t})
\leq \limsup_{t\to\infty}\int_{\underline \theta}^\infty \overline u(t,x+ x_{m,t},
\theta)d\theta
\leq \limsup_{t\to\infty} \int_{\underline \theta}^\infty \overline u(t,x_{m,t},
\theta)d\theta
\leq m.
\]
Noticing that $x_{m,t}/ t^{3/2}$ tends to $4/3$ as $t\to+\infty$ 
finishes the proof.
\end{proof}

\section{The lower bound}\label{sec:lower}

In this section, we obtain a lower bound on the propagation in the local and 
non-local cane toads equations. As we have mentioned, the idea is to construct
sub-solutions of the linearized problem with the Dirichlet boundary condition
on a moving boundary of a domain $E(t)$, 
and then use them to deduce a lower bound on the solution
of the nonlinear problem. The goal is to have $E(t)$ move as fast as possible
while ensuring that the solution of the linearized problem is $O(1)$ -- it neither grows too much, nor decays.
This strategy is inspired by the proof of the Freidlin-G\"artner formula for the standard KPP equation
by J.-M. Roquejoffre~\cite{RR-Toulouse}; however, in contrast, the coefficients that arise in our formulation are non-periodic and so new estimates are required.

To this end, given some constants~$x_c$,~$\theta_c$, and $r$,
we define the ellipse
\begin{equation}\label{eq:ellipse}
E_{x_c,\theta_c,R} 
:{=}\left\{ (x,\theta)\in \R\times[\underline\theta,\infty]: \frac{(x - x_c)^2}{\theta_c} + (\theta - \theta_c)^2 \leq R^2\right\}.
\end{equation}
Given a large time $T$, we will move such an ellipse along some trajectories $X_T(t)$ and $\Theta_T(t)$ 
on the time interval $[0,T]$, starting at a point $(X_T(0),\Theta_T(0))$. We will denote $\Theta_T(0)=H\ge\unth$.
Note that the equation is translationally invariant in $x$, so the starting 
point $X_T(0)$ is not important. 
The trajectories will satisfy certain conditions: first, they move 
``up and to the right":
\begin{equation}\label{dec404}
\dot X_T(t)\ge 0,~~\dot \theta_T(t)\ge 0,~~\hbox{ for all $0\le t\le T$.}
\end{equation}
Next, for all $0 \le t \le T$, with some fixed $\eps>0$ we assume that either
\begin{equation}\label{dec406}
L(X_T(t),\theta_T(t),\dot X_T(t),\dot\theta_T(t))\le 1-2\eps
	~~~\text{ or }~~~
	\dot X(t) = 0 = \dot \Theta(t).
\end{equation}
Here, $L(x,\theta,v_x,v_\theta)$ is the Lagrangian given by (\ref{dec402}):
\begin{equation}\label{dec408}
L(x,\theta,v_x,v_\theta) = \frac{v_x^2}{4\theta} + \frac{v_\theta^2}{4}.
\end{equation}
Finally, we assume that
\begin{equation}\label{dec410}
\lim_{T,H\to\infty} \left[\big| \ddot\Theta_T(t)\big| + 
\frac{\big|\ddot{X}_T(t)\big|}{\sqrt{\Theta_T(t)}}\right] = 0, 
~~~\text{uniformly in  $t\in[0,T]$}.
\end{equation}

We now state our main lemma, which we use in both the local and non-local settings.
\begin{lemma}\label{lem:trajectories}
Consider any trajectories $X_T(t)$ and $\Theta_T(t)$ on $[0,T]$ which satisfy
the above assumptions, and fix $\eps>0$, $\delta>0$, and $\gamma>0$ sufficiently small. There exist constants $R_{\eps}$, $T_{\eps,\delta,\gamma}$, and $H_{\eps,\gamma}$ such that for all $R\geq R_\eps$, $T\geq T_{\eps,\delta,\gamma}$, 
and $H \geq H_{\eps,\gamma}$, there is a function $v$  which solves
\begin{equation}\label{eq:trajectories_subsolution}
\begin{cases}
v_t - \theta v_{xx} - v_{\theta\theta} \leq (1-\eps) v, & \hfill (t,x,\theta) \in \R^+ \times E_{X_T(t),\Theta_T(t),R} \, ,\medskip\\
v(t,x,\theta) \leq 1,	& \hfill (t,x,\theta) \in [0,T] \times E_{X_T(t),\Theta_T(t),R} \,,\medskip \\
v(t,x,\theta) = 0, 		& \hfill (t,x,\theta) \in [0,T] \times \partial E_{X_T(t),\Theta_T(t),R} \, ,\medskip\\
v(0,x,\theta) < \delta, 		& \hfill (x,\theta) \in E_{X_T(0),\Theta_T(0), R}\, ,
\end{cases}
\end{equation}
and  such that $\|v(T,\cdot,\cdot)\|_{L^\infty} \geq 1 - \gamma$, and 
$v(T,x,\theta) \geq C_R$ for all  $ (x,\theta) \in E_{X_T(T),\Theta_T(T),R/2},
$
with a constant $C_R>0$ that depends only on $R$.
The constant $R_\eps$ depends only on $\epsilon$, the constant $H_{\eps,\gamma}$ depends only on $\eps$ and $\gamma$, and $T_{\eps,\delta,\gamma}$ depends only on $\eps$, $\delta$, $\gamma$, and the 
rate of the limit in \eqref{dec410}.
\end{lemma}
\noindent
We apply Lemma~\ref{lem:trajectories} as follows.  First, we use it to build a sub-solution  along a sufficiently large ellipse 
moving along a suitably chosen trajectory $(X_T(t),\Theta_T(t))$.  In this step, we choose $\delta$ such that we may fit $\eps v$ 
underneath the solution $u$ so that $u$ always stays above $\eps v$.  Hence, $u$ is at least of the size $\epsilon C_R$
near the point $(X_T(T),\Theta_T(T))$, after time $T$.  Then we 
re-apply the lemma, with the trivial trajectory that remains fixed at the 
point $(X_T(T),\Theta_T(T))$ for all time, to build a sub-solution to~$u$ on the ellipse $E_{X_T(T),\Theta_T(T),R/2}$ that grows from $\eps C_R$ 
to any prescribed height~$m\in(0,1)$ in $O(1)$ time, depending on $\epsilon$ and $m$.  It follows that $u$ is at least of height $m$ near $(X_T(T), \Theta(T))$ after 
time~$T + O(1)$.

The proof of Lemma~\ref{lem:trajectories} involves   estimates of the solution to a spectral problem 
posed on the moving domain $E_{X_T(t),\Theta_T(t),R}$ after a suitable change of variables and a rescaling.  
We prove this lemma in Section \ref{sec:trajectories} below.

\subsection{The lower bound in the local equation}\label{ss:local_lower}

Here, we show that Lemma~\ref{lem:trajectories} allows us to propagate a constant amount of mass along trajectories that we choose carefully.  Our main result in this section is the following.
\begin{prop}\label{prop:lower_bound}
Suppose that $u$ satisfies the local cane toads equation \eqref{nov1012} with any initial data~$u_0> 0$.  
Then, for any $m \in (0,1)$,  we have
\[
\frac43 \leq \liminf_{T\to\infty} \frac{\max \left\{ x : \exists \theta \in \Theta,  u(T,x,\theta) \geq m \right\}}{T^{3/2}}.
\]
\end{prop}\noindent
The assumption that $u_0$ is positive is not restrictive since any solution with a non-zero, non-negative initial condition 
becomes positive for all $t>0$ as a consequence of the maximum principle.  In particular, as 
the initial condition $u_0$ in Theorem~\ref{thm:local_toads}
  is compactly supported, non-negative, and non-zero,  we may apply Proposition~\ref{prop:lower_bound} to $\tilde u$, 
the solution to the cane toads equation with the initial condition $\tilde u_0(x,\theta) = u(t=1,x,\theta)$.

Take any function $v$ given by~\Cref{lem:trajectories}.  Since $v \leq 1$, then $\epsilon v \leq \epsilon$ and, hence, satisfies
\[
	(\epsilon v)_t - \theta (\epsilon v)_{xx} - (\epsilon v)_{\theta\theta} - (\epsilon v)(1 - (\epsilon v))
		\leq -\epsilon(\epsilon v) + (\epsilon v)^2
		\leq 0.
\]
In particular, if we choose $\delta$ and our trajectories well, we will have that
\begin{equation}\label{dec702}
u(t,x,\theta)\ge \eps v(t,x,\theta)\hbox{ for all $0\le t\le T$ and $(x,\theta)\in E_{X_T(t),\Theta_T(t),R}$.}
\end{equation}
Here, $u(t,x)$ is the solution to \eqref{nov1012}. In particular, we have
\begin{equation}\label{dec720}
\forall (x,\theta)\in E_{X_T(T),\Theta_T(T),R/2}, \qquad u(T,x,\theta)\ge \eps C_R,
\end{equation}
after a sufficiently long time $T$.  
However, we do not have control on the constant $C_R$.  To remedy this, we again apply \Cref{lem:trajectories}, getting 
a sub-solution $v'$ for $u$, in order to show that we can quickly grow the solution 
from this small constant to $O(1)$.  As such, we make the choices 
$R' = R/2$, and~$\delta' = \epsilon C_R$, $\epsilon' = m$.

\subsubsection*{Proof of Proposition \ref{prop:lower_bound}}

Let us now present the details of the argument.
We fix $\eps > 0$ and any $m \in (0,1)$, and  let $u$ be the solution to \eqref{nov1012} with
the  initial condition~$u_0$.  Given $T$, $R$, and $H$ to be determined later, we will use the trajectories 
which are a slightly slowed down version of the optimal 
Hamilton-Jacobi trajectories~(\ref{dec706}):
\begin{eqnarray}\label{dec708}
&&X_T(t) = (1 - 2\eps)^{3/4}\left( \frac{2t^2}{T^{1/2}} - \frac{2t^3}{3T^{3/2}}\right),\\
&&\Theta_T(t) = (1 - 2\eps)^{1/2}\left( 2t - \frac{t^2}{T}\right) + H.\nonumber
\end{eqnarray}
It is straightforward to verify to that $X_T$ and $\Theta_T$ satisfy the assumptions above Lemma~\ref{lem:trajectories}: for all $0\le t \le T$, we have
$\dot X_T(t)\ge 0$, $\dot\Theta_T\ge 0$ and 
\[
L(X_T(t),\Theta_T(t),\dot X_T(t),\dot\Theta_T(t)) \leq 1-2\eps
\]
while 
\[
\ddot X_T(t)=\farc{4(1-2\eps)^{3/4}}{T^{1/2}}\Big(1-\farc{t}{T}\Big),~~
\ddot\Theta_T(t)=-\farc{2(1-2\eps)^{1/2}}{T},
\]
so that (\ref{dec406}) and (\ref{dec410}) hold as well.

We set $\gamma = \epsilon$ and
\[
\delta \stackrel{\rm def}{=} \inf_{E_{X_T(0),\Theta_T(0),R}} u_0(x,\theta).
\]
Note that $\delta$ depends on $R$ and $H$ but not on $T$.
Applying Lemma~\ref{lem:trajectories}, we may find $R_\eps$, $T_{\eps,\delta}$ and~$H_\eps$ such that if 
$R \geq R_{\eps}$, $T \geq T_{\eps,\delta}$ and $H \geq H_{\eps}$, then there exists a function 
$v$ which satisfies \eqref{eq:trajectories_subsolution}.  Hence, as we have discussed,
the function $\eps v$ is a sub-solution to 
$u$ on $E_{X_T(t), \Theta_T(t), R}$ for all $t \in [0,T]$, and (\ref{dec702})-(\ref{dec720}) hold.  In particular, we have that
\[
	u(T,x,\theta) \geq \epsilon v(T,x,\theta) \geq \epsilon C_R
\]
for all $(x,\theta) \in E_{X_T(T),\Theta_T(T),R/2}$.

Next, we use Lemma~\ref{lem:trajectories} a second time, with the new choices
\[
	R' = R/2,
	~~~ \delta' = \frac{\epsilon C_R}{m},
	~~~ \epsilon' = \frac{1+m}{2}
	~~~ \text{ and } ~~~ \gamma' = \frac{1-m}{1+m},
\]
to find $H_m$, $R_m$ and $T_{m,\delta'}$ such that if $R/2 = R' \geq R_m$ and $H \geq H_m$ then we may find a solution $w$ to~\eqref{eq:trajectories_subsolution} on $[0,T_{m,\delta'}]\times E_{X_T(T),\Theta_T(T),R'}$.  We shift in time and scale to define
\[
	\tilde w(t,x,\theta) \stackrel{\rm def}{=} \epsilon' w(t - T, x, \theta).
\]
By our previous work and our choice of $\delta'$, it follows that $\tilde w(T,x,\theta) \leq u(T,x,\theta).$  In addition, the partial differential equation for $w$,~\eqref{eq:trajectories_subsolution}, with our choice of $m$, guarantees that $\tilde w$ is a sub-solution to $u$ on $[T,T+T_{m,\delta'}]\times E_{X_T(T),\Theta_T(T),R'}$.  This implies that
\begin{equation}\label{eq:big_enough}
\begin{split}
	\|u&(T+T_{m,\delta'},\cdot,\cdot)\|_{L^\infty(E_{X_T(T),\Theta_T(T),R'})}
		\geq \|\tilde w(T+T_{m,\delta'},\cdot,\cdot)\|_{L^\infty(E_{X_T(T),\Theta_T(T),R'})}\\
		&= \epsilon' \|w(T_{m,\delta'}, \cdot,\cdot)\|_{L^\infty(E_{X_T(T),\Theta_T(T),R'})}
		\geq \frac{1+m}{2} \left(1 - \gamma\right)
		= \frac{1+m}{2} \left(1 - \frac{1-m}{1+m}\right)
		\geq m.
\end{split}
\end{equation}
The first inequality is a consequence of the fact that $\tilde w$ is a sub-solution of $u$, the first equality is a consequence of the definition of $\tilde w$, and the final equalities is a consequence of~\Cref{lem:trajectories} and the choices of $\epsilon'$ and $\gamma$.

The above,~\eqref{eq:big_enough}, implies that $u$ achieves a value at least as large as $m$ for some 
\[
x \geq X_T(T) - \sqrt{\Theta_T(T)}R'
\]
at time $T + T_{m,\delta'}$. As a consequence, we have
\begin{equation}\label{dec722}
\frac{4}{3}(1-2\eps)^{3/4}
  \left(\frac{T}{T+T_{m,\delta'}}\right)^{3/2} - \frac{\sqrt{T + H}}{(T+T_{m,\delta'})^{3/2}}R' 
\leq \frac{\max\{x: \exists \theta \in \Theta, u(T+T_{m,\delta}, x, \theta) =m\}}{(T+T_{m,\delta'})^{3/2}}.
\end{equation}
Here, we used the definition of $X_T(T)$ and $\Theta_T(T)$, and the upper bound 
\[
\Theta_T(T) \leq T + H+R.
\]
Since (\ref{dec722}) holds for all $T$ sufficiently large,
and since $H$, $R$, and $T_{m,\delta}$ are fixed, we may take the limit as $T$ tends to infinity to obtain
\[
\frac{4}{3}(1 - 2\eps)  ^{3/4}
\leq \liminf_{T\to\infty} \frac{\max\{x: \exists \theta \in \Theta, u(T, x,\theta) \geq m\}}{T^{3/2}}.
\]
Since $\eps$ is arbitrary, this finishes the proof of Proposition \ref{prop:lower_bound}.~$\Box$

\subsection{The lower bound in the non-local equation}\label{ss:nonlocal_lower}

In this section we prove the following proposition.
\begin{prop}\label{prop:nonlocal_lowerbound}
Suppose that $n$ is a solution of the non-local cane toads equation~\eqref{nov1016}
with a positive initial condition $n_0\in L^\infty(\R \times \Theta)$.  Define
\[
	\rho_0(x) \stackrel{\rm def}{=} \int_{\underline \theta}^\infty n_0(x,\theta)d\theta,
\]
and assume that $\rho_0 \in L^\infty(\R)$.  Then, for any $\eps>0$ and $\gamma \in (0,1)$ we have
\begin{equation}\label{dec1802}
c_1\stackrel{\rm def}{=}\frac{8}{3\sqrt{3\sqrt{3}}}\left(1 - 2\eps\right)^{3/4} \leq \limsup_{T\to\infty} \frac{\max \left\{ x : \rho(T,x) \geq \gamma \eps
\right\}}{T^{3/2}}.
\end{equation}
\end{prop}

Our strategy here is similar to the local case, though this time we argue by contradiction.  
Suppose that the result does not hold.  Then there exists $\eps>0$, a time $t_\eps > 0$ and $\gamma \in (0,1)$ such that, for all $t\geq t_\eps$,

\begin{equation}\label{eq:rho_bound}
	\sup_{x \geq c_1 t^{3/2}} \rho(t,x) < \gamma\eps.
\end{equation}

Our goal is  to construct a sub-solution $v$ to $n$ which will satisfy
\[
\int_{\unth}^\infty v(t,x,\theta)d\theta\ge \eps>\gamma\eps,
\]
for some $x\ge c_1t^{3/2}$.
This will push $\rho(t,x)$ 
to be greater than $\eps$ as well, yielding a contradiction to~\eqref{eq:rho_bound}.  
Note that if \eqref{eq:rho_bound} holds, then any solution to
\begin{equation}\label{dec902}
	v_t - \theta v_{xx} - v_{\theta\theta} \leq v ( 1 - \gamma\eps),
\end{equation}
defined for $t \geq t_\eps$ and which is supported on $x\ge c_1 t^{3/2}$ 
satisfies
\begin{equation}\label{dec1402}
n(t,x,\theta)\ge v(t,x,\theta)\hbox{ for all $t\ge t_\eps$,
$(t,x,\theta)\in\hbox{supp}~v$},
\end{equation}
provided that this inequality holds at $t=t_\eps$.
This is because
(\ref{eq:rho_bound}) implies
\begin{equation}\label{dec904}
	n_t - \theta n_{xx} - n_{\theta\theta}
		= n(1-\rho)
		\geq n (1 - \gamma\eps),
\end{equation}
for all $t \geq t_\eps$ and $x \geq c_1 t^{3/2}$.
Note that the nature of the argument by contradiction requires us now to have the ``Dirichlet ball" completely to the right of  
$x=c_1t^{3/2}$. On the other hand, the ``nearly optimal" Hamilton-Jacobi trajectories (\ref{dec708}) that we have used in the local case,
initially move mostly in the $\theta$-direction when $T$ is large, and violate this condition.  This forces us to choose sub-optimal
trajectories for the center of the ``Dirichlet ball", leading to the non-sharp constant $c_1$ in~(\ref{dec1802}). 
We assume now (\ref{eq:rho_bound})  and
define the trajectories
\begin{equation}\label{dec1404}
X_T(t) = c_\gamma(t+t_\eps)^{3/2},~~~~
\Theta_T(t) = (1-2\gamma\eps)^{1/2} \frac{2}{\sqrt{3}}(t+t_\eps) + H,
\end{equation}
with
\[
c_\gamma  {=} \frac{8}{3\sqrt{3\sqrt{3}}} (1 - 2\gamma\eps)^{3/4}.
\]
As $c_\gamma>c_1$, we have 
\[
\hbox{$X_T(t)>c_1(t+t_\eps)^{3/2}$ for all $t>0$.}
\]
The constant $H$ will be determined below.  We note that $X_T$ and $\Theta_T$ satisfy the assumptions preceding 
Lemma~\ref{lem:trajectories}.  
Indeed, both the non-negativity of $\dot X_T$ and $\dot \Theta_T$ in (\ref{dec404})
and  the limit in (\ref{dec410}) are obvious from (\ref{dec1404}).   
Hence, we only need to check the  condition on the Lagrangian in \eqref{dec406}.  To this end, we compute
\begin{eqnarray*}
&&\frac{\dot X_T^2}{4\Theta_T} + \frac{\dot \Theta_T^2}{4}
=\farc{4(1-2\gamma\eps)(t+t_\eps)}{
6(t+t_\eps)+\frac{3\sqrt{3}}{\sqrt{1-2\gamma\epsilon}}H}+\frac{(1-2\gamma\eps) }{3}
		< (1 - 2\gamma\eps).
\end{eqnarray*}

We now build a sub-solution on $E_{X_T(t),\Theta_T(t),R}$ 
for suitably chosen $H$, $R$, and $T$ which grows and forces $\rho$ to be larger than $\gamma\eps$, giving us a contradiction.  
The aforementioned condition on the support, i.e.~that $(x,\theta)\in\supp(v)$ only if $x \geq c_1t^{3/2}$,  is equivalent to
\begin{equation}\label{eq:subsolution_condition}
	X_T(t-t_\eps) - R\sqrt{\Theta_T(t-t_\eps)} \geq c_1t^{3/2},
\end{equation}
for all $t \in[t_\eps, t_\eps + T]$.  In other words, the left-most point of a ball centered at $X_T(t-t_\eps)$ of the radius~$R\sqrt{\Theta_T(t-t_\eps)}$ must be 
to the right of $c_1t^{3/2}$.  Since $c_\gamma > c_1$, 
we can clearly arrange for this to be satisfied by increasing, if necessary, 
$t_\eps$ in a way depending only on $R$ and $H$.
 
Fix $M > 0$ to be determined later.  Let us define
\[
\delta \stackrel{\rm def}{=} \frac{1}{M} \inf\limits_{E_{X_{T}(0),\Theta_{T}(0),R}} n(t_\eps, x, \theta).
\]
Note that $\delta$ depends on $R$ and $M$ but not on $T$.  Applying Lemma~\ref{lem:trajectories} with $\epsilon$ and $\delta$ defined above and $\gamma = 1/2$, we may find 
$H_\eps$ and $R_\eps$, depending only on $\eps$, and 
$T_{\eps,\delta}$, that depends only on $\eps$, $\delta$, such that, 
if $H \geq H_\eps$, $R \geq R_\eps$, and $T\geq T_{\eps,\delta}$, we may find $v$ which satisfies~\eqref{eq:trajectories_subsolution}.  Define
\[
	\tilde v(t,x,\theta) = M v(t - t_\epsilon, x, \theta).
\]
By virtue of the discussion above and by~\eqref{eq:trajectories_subsolution}, we see that $\tilde v$ is a sub-solution of $u$ on $[t_\epsilon, t_\epsilon + T_{\epsilon,\delta}]\times E_{X_T(t),\Theta_T(t),R}$.    In particular, we have that
\begin{equation}\label{eq:u_lower_bound}
\begin{split}
	\inf_{E_{X_T(T),\Theta_T(T),R/2}}u(t_\epsilon + T, \cdot, \cdot)
		&\geq \inf_{E_{X_T(T),\Theta_T(T),R/2}} \tilde v(t_\epsilon + T, \cdot,\cdot)\\
		&= \frac{M}{2} \inf_{E_{X_T(T),\Theta_T(T),R/2}} v(t_\epsilon + T,\cdot,\cdot)
		\geq \frac{M}{2} C_R.
\end{split}
\end{equation}
We emphasize here that $C_R$ depends only on $R$ and not on $M$.
At the expense of possibly increasing $T_{\epsilon,\delta}$, we can now specify $M= \frac{2\eps}{R C_R}$.  Using now~\eqref{eq:u_lower_bound}, we obtain
\[\begin{split}
\int_{\unth}^\infty u(t_\eps+T,X_T(T),\theta)d\theta
	&> \int_{\Theta_T(T) - R/2}^{\Theta_T(T)+R/2} u(t_\eps+T,X_T(T),\theta)d\theta\\
	&\geq \int_{\Theta_T(T) - R/2}^{\Theta_T(T)+R/2} M C_R d\theta
	= MC_RR = 2\eps.
\end{split}\] 
Hence we have $\rho(t_\eps + T,X_T(T)) > \eps$.  As we have
\[
X_T(T)>c_1(t_\eps+T)^{3/2},
\]  
this contradicts~\eqref{eq:rho_bound}, finishing the proof.~$\Box$

\section{Proofs of the auxiliary lemmas}\label{sec:trajectories}

In this section we prove the auxiliary results needed in the construction of the sub-solutions.
Some of them are quite standard, we present the proofs for the convenience of the reader.

\subsection{Existence of a sub-solution along trajectories -- the proof of
Lemma \ref{lem:trajectories}}

In this subsection we prove  Lemma \ref{lem:trajectories} by suitably re-scaling the equation and then using careful spectral estimates.  
Recall that our goal is to show that 
there exist constants $R_\eps$, $T_{\eps,\delta,\gamma}$, and $H_{\eps,\gamma}$ such that for all $R\geq R_\eps$, $T\geq T_{\eps,\delta,\gamma}$, 
and $H \geq H_{\eps,\gamma}$, there is a function $v$  which satisfies
\begin{equation}\label{eq:trajectories_subsolutionbis}
\begin{cases}
v_t - \theta v_{xx} - v_{\theta\theta} \leq (1-\eps) v, ~~~~~&\text{ for all $t>0$,  and $(x,\theta) \in E_{X_T(t),\Theta_T(t),R}$},\medskip\\
v(t,x,\theta) = 0, 		&\text{ for all  $0\le t\le T$ and } (x,\theta) \in \partial E_{X_T(t),\Theta_T(t),R},\medskip\\
v(0,x,\theta) < \delta, 		&\text{ for all } (x,\theta) \in E_{X_T(0),\Theta_T(0), R},\medskip\\
v(t,x,\theta) \leq 1,	&\text{ for all $0\le t\le T$ and } (x,\theta)\in E_{X_T(t),\Theta_T(t),R},
\end{cases}
\end{equation}
and  such that $\|v(T,\cdot,\cdot)\|_{L^\infty} \geq 1-\gamma$, and 
$v(T,x,\theta) \geq C_R$ for all  $ (x,\theta) \in E_{X_T(T),\Theta_T(T),R/2},
$
with a constant $C_R>0$ that depends only on $R$ and $\delta>0$.
To construct the desired sub-solution, we first go into the moving frame, and rescale the spatial variable:
\begin{equation}\label{eq:v}
v(t,x,\theta) = \tilde v\left( t, \frac{x - X_T}{\sqrt{\Theta_T}}, \theta - \Theta_T\right),~~~
y = \frac{x - X_T}{\sqrt{\Theta_T}}, ~~\text{ and }~~ \eta = \theta - \Theta_T.
\end{equation}
Then~\eqref{eq:trajectories_subsolutionbis} yields  
\begin{equation}\label{eq:v_inequality}
	\tilde v_t  - \left(\frac{y}{2} \frac{\dot\Theta_T}{\Theta_T} + \frac{\dot{X}_T}{\sqrt{\Theta_T}} \right)\tilde v_y  - \dot\Theta_T \tilde v_\eta \leq \left(1 + \frac{\eta}{\Theta_T}\right) \tilde v_{yy} + \tilde v_{\eta\eta} + (1 - \eps) \tilde v,
\end{equation}
with the boundary conditions  
\[
	\tilde v(t,y,\eta) = 0, ~~~~~~\text{ for all } (y,\eta) \in \partial B_R. 
\]
Here, $B_R\stackrel{\rm def}{=}B_R(0,0)$ is a ball of radius $R$ centered at $(y,\eta) = (0,0)$.   
As in~\cite{HNRR13,HNRR12}, the next step is to remove an exponential, setting, 
\[
w(t,y,\eta) = e^{ \frac{\dot{X}}{2\sqrt{\Theta_T}}y + \frac{\dot\Theta_T}{2} \eta} \tilde v(t,y,\eta).
\]
Note that if $T$ and $H$ are sufficiently large, there is a constant $M_R$, depending only on $R$, such that
\begin{equation}\label{eq:comparability}
	\frac{1}{M_R} w(t,y,\eta) \leq \tilde v(t,y,\eta) \leq M_R w(t,y,\eta)
\end{equation}
holds for all $t$, $y$, and $\eta$, because of the uniform bound (\ref{dec406}) on the Lagrangian. 
Changing variables in~\eqref{eq:v_inequality}, we see that $w$ must satisfy the inequality

\[\begin{split}
	w_t - &\underbrace{\Big(\frac{y \dot\Theta_T}{2\Theta_T} - \frac{\dot{X_T}}{\sqrt{\Theta_T}}\frac{\eta}{\Theta_T}\Big)}_{\stackrel{\text{def}}{=} A} w_y
		\leq \underbrace{\Big(1 + \frac{\eta}{\Theta_T}\Big)}_{\stackrel{\text{def}}{=} 1+D} w_{yy} + w_{\eta\eta}\\
			& + w\Big( \underbrace{1 -\eps- \frac{\dot{X}_T^2}{4\Theta_T} - \frac{\dot{\Theta}_T^2}{4} + 
			\Big( \frac{\ddot{X}_T}{2\sqrt{\Theta_T}} - \frac{\dot{X}_T\dot{\Theta}_T}{4\Theta_T^{\frac32}} \Big) y 
			+ \Big( \frac{\dot{X}_T^2}{4\Theta_T^2} + \frac{\ddot{\Theta}_T}{2} \Big) \eta }_{\stackrel{\text{def}}{=} G}\Big).
\end{split}
\label{eq:trajectories}
\]
Note that by choosing $T$ and $H$ large enough and using \eqref{dec406} and (\ref{dec410}), we may ensure that
\[
	\left|\Big( \frac{\ddot{X}_T}{2\sqrt{\Theta_T}} - \frac{\dot{X}_T\dot{\Theta}_T}{4\Theta_T^{\frac32}} \Big) y 
			+ \Big( \frac{\dot{X}_T^2}{4\Theta_T^2} + \frac{\ddot{\Theta}_T}{2} \Big) \eta\right| \leq \min\left\{\frac\epsilon4, \frac{1-\eps}{2} \right\}.
\]
Hence, using (\ref{dec406}), we have that
 \[
G\ge 1 -\eps- \frac{\dot{X}_T^2}{4\Theta_T} - \frac{\dot{\Theta}_T^2}{4}-\min\left\{\frac{\epsilon}{4},\frac{1-\eps}{2}\right\}
	\ge\min\left\{\frac{1-\epsilon}{2},\frac{3\eps}{4}\right\}.
\]
The two cases in the minimum correspond to the two cases in~\eqref{dec406}.
Hence,  if we construct $w$ which satisfies
\begin{equation}\label{eq:key_differential_inequality}
	w_t - A w_y \leq (1+D)w_{yy} + w_{\eta\eta} + \min\left\{\frac{1-\epsilon}{2},\frac{3\eps}{4}\right\}w.
\end{equation}
then $w$ also satisfies
\[
	w_t - A w_y \leq (1+D) w_{yy} + w_{\eta\eta} + G w.
\]
Returning to the original variables, $v$ would satisfy the desired differential inequality.  
With this in mind, we seek to construct $w$ satisfying~\eqref{eq:key_differential_inequality} which has the desired bounds.

We define $w$ using the principal Dirichlet eigenfunction of the operator
\begin{equation}\label{eq:L}
	L_{T,H}(t) \stackrel{\text{def}}{=} - A \partial_y - (1+D) \partial_{yy} - \partial_{\eta\eta}.
\end{equation}
To this end, for each $t \in [0,T]$, define $\phi_{T,H}(t,x,\theta)$ and $\lambda_{T,H}(t)$ to be the 
principal Dirichlet eigenfunction and eigenvalue of $L_{T,H}(t)$ (depending on $t$ as a parameter) in the ball $B_R$,
with the normalization 
\[
\|\phi_{T,H}\|_{L^2(B_R)} = 1.
\]
We state two lemmas regarding these quantities which will allow us to finish the proof.  First, 
we need to understand the behavior of $\lambda_{T,H}$ for $T$ and $H$ large. We recall the following result.
\begin{lemma}\label{lem:eigenvalue}
Fix $s > 1$.  Consider the operator
\[
	L_{a,b} = - \nabla \cdot a \nabla - b \cdot \nabla
,\]
defined on a smooth, bounded domain $\Omega \subset \R^d$ with $a,b\in C^{s}(\Omega)$ and where 
$a$ is a uniformly positive definite matrix.  Let $\lambda_{a,b,\Omega}$ be the principal Dirichlet eigenvalue of $L_{a,b}$ with eigenfunction $
\phi_{a,b,\Omega}$ having $L^\infty$-norm one.  Then $\lambda_{a,b,\Omega}$ and $\phi_{a,b,\Omega}$ are continuous in $a$ and $b$, when 
considered as maps from $(L^\infty)^{d^2}\times (L^\infty)^d$ to $\R$ and to $H^{s}$, respectively.
\end{lemma}

\noindent The proof of Lemma~\ref{lem:eigenvalue} will be presented in \cref{sec:spectrum} below.

\Cref{lem:eigenvalue} implies that, as $T$ and $H$ tend to infinity, $\lambda_{T,H}(t)$ becomes bounded above and below by a constant multiple 
of $R^{-2}$ since the principal eigenvalue of $-\Delta$ on $B_R$ equals $c_1/R^2$ for some positive constant $c_1$.  This convergence 
is uniform in $t$.  Hence, choosing first $R$ sufficiently large, we may choose $H$ and $T$, depending only on 
$\eps$, $R$, the convergence rate of the limit in \eqref{dec410}, such that 
\begin{equation}\label{eq:small_eigenvalue}
	\lambda_{T,H}(t) < \frac{1}{4}\min\left\{\frac{1-\epsilon}{2},\frac\eps4\right\},~~~~\text{ for all } t.
\end{equation} 
We will also need   the behavior of the time derivative of $\phi_{T,H}$.
\begin{lemma}\label{lem:time_derivative}
Using the notation above, $\partial_t \phi_{T,H}$ is a smooth function in $y$ and $\eta$, and
\[
	\lim_{T,H \to \infty} \left\| \frac{\partial_t \phi_{T,H}}{\phi_{T,H}}\right\|_{L^\infty} = 0.
\]
\end{lemma}
\Cref{lem:time_derivative} implies that for fixed $R$, we may choose $T$ and $H$, depending only on $\eps$ and the convergence rate of the limit in \eqref{dec410}, such that
\begin{equation}\label{eq:small_derivative}
	|\partial_t\phi_{T,H}| \leq \frac{1}{4}\min\left\{\frac{1-\epsilon}{2},\frac\eps4\right\} \phi_{T,H}, ~~~~ \text{ for all } (t,y,\eta) \in [0,T]\times B_R.
\end{equation}

With this set-up, we can now conclude the proof of Lemma~\ref{lem:trajectories}.  We define
\[ 
w_{T,H} \stackrel{\rm def}{=} \delta e^{r t} \frac{\phi_{T,H}}{N} ,
\]
with
\[
r \stackrel{\rm def}{=} -\frac{1}{T} \log\left(\delta\right)
~~~\text{ and }~~~
N \stackrel{\rm def}{=} \sup_{t\in[0,T]}\|\phi_{T,H}(t)\|_\infty.
\]
Fix $T$ large enough, depending on $\eps$ and $\delta$, such that 
\[
	r < \frac{1}{4}\min\left\{\frac{1-\epsilon}{2},\frac\eps4\right\}.
\]
Then, \eqref{eq:small_eigenvalue} and \eqref{eq:small_derivative} and our choice of $r$ imply that
\[\begin{split}
	\partial_t w_{T,H} - L_{T,H} w_{T,H}
		&= \partial_t w_{T,H} + \lambda_{T,H} w_{T,H}\\
		&\leq \frac{\partial_t \phi_{T,H}}{\phi_{T,H}}w_{T,H} + r w_{T,H} + \lambda_{T,H}w_{T,H}
		\leq \frac{3}{4}\min\left\{\frac{1-\epsilon}{2},\frac\eps4\right\}w_{T,H},
\end{split}\]
and
$w_{T,H}$ is a sub-solution to~\eqref{eq:trajectories_subsolution}.  In addition, by construction, we know 
 that $w_{T,H}(0,y,\eta) \leq \delta$ for all $(y,\eta)\in B_R$,
and $w_{T,H}(t,y,\eta) \leq 1$ for all $t\in [0,T]$ and $(y,\eta)\in B_R$.

Finally, due to the uniform convergence of $\phi_{T,H}$ to the principal Dirichlet eigenfunction of $B_R$, if $T$ and $H$ are sufficiently large, there exists $c_R$ such that
\[
	 c_R \leq \min_{(y,\eta) \in B_{R/2}} \frac{\phi_{T,H}(y,\eta)}{N}
		= \min_{(y,\eta)\in B_{R/2}} w_{T,H}(T,y,\eta).
\]
In view of~\eqref{eq:comparability} and by undoing the change of variables, this implies
\[
	 c_R \leq \min_{(y,\eta) \in B_{R/2}} w_{T,H}(T,y,\eta)
		\leq M_R \min_{(x,\theta)\in E_{X_T(T),\Theta_T(T),R/2}} v(T,x,\theta),
\]
which finishes the proof of the lower bound on $\tilde v$.

The fact that $\|\tilde v \|_\infty \geq 1 - \gamma$ follows by noting the following:  $\phi_{T,H}/N$ converges to the Dirichlet eigenfunction with $L^\infty$ norm 1, which takes a maximum of $1$ at $(y,\eta) = (0,0)$.  Choosing $T$ and $H$ large enough, we have that $\phi_{T,H}(T,0,0)/N \geq 1 - \gamma$.  Using the definition of $w_{T,H}$ finishes the proof.~$\Box$

\subsection{The spectral problem -- the proofs of \Cref{lem:eigenvalue} and \Cref{lem:time_derivative}}\label{sec:spectrum}

\subsubsection*{The proof of \Cref{lem:eigenvalue}.}

To simplify the notation, we show continuity 
at $\lambda_{I,0,\Omega}$ and $\phi_{I,0,\Omega}$, but the same argument works 
for all matrices $a$ and vector fields $b$.  Fix any sequence $a_n$ and $b_n$ which tend to $I$ and $0$, respectively.  Let $\phi_n$ and $
\lambda_n$ be the principal eigenfunction and eigenvalue
of the operator $L_n := L_{a_n, b_n, \Omega}$, with the normalization $\Vert \phi_n \Vert_{L^2(\Omega)}= 1$.

First, we show that $\lambda_n$ is bounded.  Suppose by contradiction that (up to a subsequence that we omit)
$\lambda_n$ tends to infinity, and write
\[
\frac{1}{\lambda_n}L_n \phi_n =  \phi_n.
\]
Multiplying by a test function $\eta\in C_0^\infty(\Omega)$, integrating over $\Omega$ and 
passing to the limit $n\to+\infty$ we see that $\phi_n$ converges weakly to zero as $n\to+\infty$.
and the right side is bounded in $L^2$ uniformly in $n$, then elliptic regularity implies that $\tilde \phi_n$ is bounded uniformly in $H^2$.  Let $\tilde\phi_\infty$ be the strong $H^1$ limit of $\tilde \phi_n$, taking a subsequence if necessary.  Multiplying the equation for $\tilde\phi_n$ by $\tilde \phi_n$ and integrating, we see that
\[
	\int_\Omega \nabla\tilde \phi_n \cdot a_n \nabla \tilde\phi_n  \, dx + \int_\Omega \tilde\phi b_n \cdot \nabla \tilde\phi_n \, dx = 1.
\]
Taking the limit as $n$ tends to infinity yields
\[
	\Vert \nabla \phi_n \Vert_{L^2(\Omega)} = 1.
\]
On the other hand, since $\Vert \phi_n \Vert_{L^2(\Omega)}= 1$ and $\lambda_n$ tends to infinity, it is clear that $\tilde\phi_\infty = 0$.  This is a contradiction.

Since $\lambda_n$ is bounded uniformly, then $\phi_n$ is bounded uniformly in~$H^2$.  
Hence, we may, taking a subsequence if necessary, find a strong $H^{1+s}$ limit of $\phi_n$, $\phi_\infty$, and a limit of $\lambda_n$, $\lambda_\infty$.  Since $a_n$ and $b_n$ converge in $L^\infty$, then we see that $\phi_\infty$ must satisfy
\[
	-\Delta \phi_\infty = \lambda_\infty \phi_\infty.
\]
The maximum principle and the fact that $\phi_n > 0$ imply that $\phi_\infty > 0$.  Hence $\phi_\infty$ must be the principal eigenfunction of $-\Delta$ and $\lambda_\infty$ must be the principal eigenvalue.

We have shown that every sequence $\phi_n$ and $\lambda_n$ has a subsequence which must converge to the $\phi_{I,0,\Omega}$ and $\lambda_{I,0,\Omega}$, respectively.  This implies that $\phi_{a,b,\Omega}$ and $\lambda_{a,b,\Omega}$ are continuous at $a = I$ and $b = 0$, finishing the proof.
$\Box$

\subsubsection*{The proof of \Cref{lem:time_derivative}.}

First we show that $\partial_t \phi_{T,H}$ is well-defined.  To do this, we need only show that $\lambda_{T,H}(t)$ is Lipschitz continuous in $t$.  Indeed, with this shown, if $\lambda_{T,H}(t)$ is in $W^{1,\infty}$ as a function of $t$, we may take a derivative in $t$ of 
the equation for $\phi_{T,H}$, allowing us to write down an equation for $\partial_t \phi_{T,H}$, showing that $\partial_t \phi_{T,H}$ is well-defined.

Before we begin, we recall that we may characterize $\lambda_{T,H}$ as
\[
	\lambda_{T,H}(t) = \sup_{0<\psi \in C^2_0(B_R)} \inf_{(y,\eta)\in B_R} \frac{L_{T,H}(t) \psi(y,\eta)}{\psi(y,\eta)} = \inf_{0<\psi \in C^2_0(B_R)} \sup_{(y,\eta) \in B_R} \frac{L_{T,H}^*(t) \psi(y,\eta)}{\psi(y,\eta)}.
\]
Using this, we now estimate $\lambda_{T,H}(t+h) - \lambda_{T,H}(t)$ from below -- the upper bound may be found similarly.  
For any $h$ small enough and $p_h$ to be determined, we have that, as long as $\phi_{T,H} + h p_h > 0$,
\begin{equation}\label{eq:minimax_bound}
\begin{split}
	\lambda_{T,H}(t+h)
		&\geq \inf_{(y,\eta)} \left( \frac{L_{T,H}(t+h) \phi_{T,H}(t) + h L_{T,H}(t+h) p_h}{\phi_{T,H}(t) + h p_h}\right)\\
		&= \inf_{(y,\eta)} \left( \frac{\lambda_{T,H}(t)\phi_{T,H}(t)}{\phi_{T,H}(t) + hp_h} + \frac{(L_{T,H}(t+h) - L_{T,H}(t)) \phi_{T,H}(t) + h L_{T,H}(t+h) p_h}{\phi_{T,H}(t) + hp_h}\right).
\end{split}
\end{equation}
This suggests that we define $p_h$ as the unique solution of
\begin{equation}\label{eq:ph}
\begin{split}
&L_{T,H}(t+h) p_h = -\frac{L_{T,H}(t+h)-L_{T,H}(t)}{h} \phi_{T,H}(t),\hbox{ in $B_R$}\\
&	p_h|_{\partial B_R} = 0.
\end{split}
\end{equation}
With the following lemma, we may choose $h$ small enough that $\psi + h p_h>0$, and $\psi+hp_h$ is admissible in the formula above~\eqref{eq:minimax_bound}.
\begin{lem}\label{lem:p_h}
	For $h$ sufficiently small, $\|p_h / \phi_{T,H}(t)\|_\infty \leq C$.  Further, $\|p_h/\phi_{T,H}\|_\infty \to 0$ as $T,H \to \infty$ uniformly in $t$ and $h$.
\end{lem}

Returning to \eqref{eq:minimax_bound} and using our choice of $p_h$, we obtain the inequality:
\[
	\lambda_{T,H}(t+h)
		\geq \lambda_{T,H}(t) \inf_x\left(\frac{1}{1 + h \frac{p_h}{\phi_{T,H}(t)}} \right)
		\geq \lambda_{T,H}(t) \left(1 - C \left\|\frac{p_h}{\phi_{T,H}(t)}\right\|_{L^\infty}  h\right),
\]
with a universal constant $C$. 
Since $\lambda_{T,H}(t)$ is bounded independently of $t$, $T$, and $H$ by a constant which we may denote by $M$, we arrive at
\[
	\lambda_{T,H}(t+h) - \lambda_{T,H}(t)
		\geq - \lambda_{T,H}(t) C \left\|\frac{p_h}{\phi_{T,H}(t)}\right\|_{L^\infty} h
		\geq - CM\left\|\frac{p_h}{\phi_{T,H}(t)}\right\|_{L^\infty} h.
\]
Hence, $\lambda_{T,H}(t)$ is Lipschitz and its Lipschitz bound is linear in $\|p_h/\phi_{T,H}(t)\|_{L^\infty}$.

Since we may easily obtain a similar upper bound, we obtain
\[
	\left| \partial_t \lambda_{T,H}(t)\right| \leq CM \left\|\frac{p_h}{\phi_{T,H}(t)}\right\|_{L^\infty}.
\]
Using \Cref{lem:p_h} and the above bound implies that $\partial_t \lambda_{T,H}(t)$ tends to zero in $L^\infty([0,T])$ as $T$ and $H$ tend to infinity.

We now show that $\partial_t \phi_{T,H}$ tends uniformly to zero in $L^2$. 
  
We argue by contradiction; assume that $\|\partial_t \phi_{T,H}\|_2$ is bounded from below, passing to a subsequence if necessary.  Let us define
\[
	\psi_{T,H} \stackrel{\rm def}{=} \frac{\partial_t \phi_{T,H}}{\|\partial_t\phi_{T,H}\|_{L^2}}.
\]
Taking the $t$ derivative of the equation for $\phi_{T,H}$ yields the equation
\begin{equation}\label{eq:time_derivative_pde}
	L_{T,H} \psi_{T,H}
		-\lambda_{T,H} \psi_{T,H}= \frac{(\partial_t\lambda_{T,H}) \phi_{T,H} + (\partial_t A) \partial_y\phi_{T,H} +  (\partial_tD) \partial_{yy}\phi_{T,H}}{\|\partial_t\phi_{T,H}\|_{L^2}}.
\end{equation}
The explicit forms of $A$ and $D$ and the fact that $\partial_t \lambda_{T,H}$ shows that,  if 
$\|\partial_t \phi_{T,H}\|_{L^2}$ is bounded uniformly below, then the right hand side of this equality tends uniformly to zero in 
$L^2$ as $T$ and~$H$ tend to infinity.  In addition, calling $\lambda_{B_R}$ the principal Dirichlet eigenvalue of $-\Delta$ on $B_R$, we have that $\lambda_{T,H}$ converges to $\lambda_{B_R}$ by \Cref{lem:eigenvalue}.

By elliptic regularity, see~e.g.~\cite[Theorem~8.13]{GilbargTrudinger}, $\psi_{T,H}$ is uniformly bounded in $H^2$ since it has $L^2$ norm $1$. Up to taking a subsequence if necessary, 
we see that $\psi_{T,H}$ converges to some function $\psi$ with $\|\psi\|_{L^2} = 1$ that, owing to~\eqref{eq:time_derivative_pde}, solves
\begin{equation}\label{eq:eigenfunction_limit}
	- \Delta \psi - \lambda_{B_R} \psi = 0.
\end{equation}
Let $\phi_{B_R}$ be the $L^2$-normalized Dirichlet eigenfunction  
corresponding to $\lambda_{B_R}$.  By \Cref{lem:eigenvalue}, it follows that~$\phi_{T,H}$ converges to $\phi_{B_R}$.  Since $\lambda_{B_R}$ is a principal eigenvalue, it is simple, see e.g.~\cite[Chapter~VIII]{DautrayLions}.  Hence,

$\psi = a \phi_{B_R}$.  Note that $|a| = 1$ since $\|\psi\|_{L^2} = 1$.  On the other hand, we have that
\[
	a=\int_{B_R} \psi \phi_{B_R} dyd\eta
		= \lim_{T,H\to\infty} \int_{B_R} \left(\frac{\partial_t \phi_{T,H}}{\|\partial_t\phi_{T,H}\|_{L^2}}\right) \phi_{T,H} dyd\eta
		= \lim_{T,H\to\infty} \frac{1}{2\|\partial_t\phi_{T,H}\|_{L^2}} \partial_t \|\phi_{T,H}\|^2
		= 0.
\]
The last equality holds since   $\|\phi_{T,H}\|_{L^2}=1$ for all $t$.  
This contradicts the fact that $|a| = 1$.
Hence, it must be that $\|\partial_t \phi_{T,H}\|_{L^2}$ tends to zero uniformly as $T$ and $H$ tend to infinity.

Knowing that $\partial_t \phi_{T,H}$ and $\partial_t \lambda_{T,H}$ tend to zero, we may now conclude.  First, we note that
\[
	L_{T,H} (\partial_t\phi_{T,H})
		-\lambda_{T,H} (\partial_t\phi_{T,H})= (\partial_t\lambda_{T,H}) \phi_{T,H} + (\partial_t A) \partial_y\phi_{T,H} + (\partial_t D) \partial_{yy} \phi_{T,H}.
\]
Since the right side tends to zero in $H^2$ and since $\partial_t \phi_{T,H}$ tends to zero in $L^2$, it follows that $\partial_t \phi_{T,H}$ tends to zero in $H^4$, see e.g.~\cite[Theorem~8.13]{GilbargTrudinger}.  Using Morrey's inequality~\cite[Theorem~7.22]{GilbargTrudinger}, we may strengthen this to show that $\partial_t \phi_{T,H}$ converges to zero in $C^1(B_R)$ uniformly in $t$.  On the other hand, by \Cref{lem:eigenvalue} $\phi_{T,H}$ converges in $H^4$ to $\phi_{B_R}$.  Again, using Morrey's inequality, we have that $\phi_{T,H}$ converge to $\phi_{B_R}$ in $C^1(B_R)$ uniformly in $t$.

It follows from this $C^1$ convergence that when $T$ and $H$ are sufficiently large, $\phi_{T,H}$ is uniformly positive 
for any compact subset of $B_R$ and $\partial_n \phi_{T,H}$ is uniformly negative, where $n$ is the outward  
unit normal of $\partial B_R$ since $\phi_{B_R}$ has these properties.
On the other hand $\partial_t \phi_{T,H}$ converges uniformly to zero on $B_R$ 
and $\partial_n (\partial_t \phi_{T,H})$ converges uniformly to zero on $\partial B_R$.  This yields
\[
	\lim_{T,H \to \infty} \left\|\frac{\partial_t \phi_{T,H}}{\phi_{T,H}}\right\|_{L^\infty} = 0,
\]
which finishes the proof.~$\Box$

In order to conclude, we must now prove \Cref{lem:p_h}.  We do that here.
\begin{proof}[{\bf Proof of \Cref{lem:p_h}}]
Since $\|\phi_{T,H}(t)\|_{L^2(B_R)} = 1$ and since the coefficients of $L_{T,H}$ are bounded in $C^2$, then estimates as in, e.g.~\cite[Theorem~8.10 and Theorem~8.13]{GilbargTrudinger}, imply that $\|\phi_{T,H}(t)\|_{H^4(B_R)} \leq C_R$, for a constant independent of $T$ and $H$.  

As the coefficients in $L_{T,H}(t)$ are have two bounded derivatives in $y$ and $\eta$, we may apply these same estimates to obtain that 
\[\begin{split}
	\|p_h\|_{H^4(B_R)}
		&\leq C\left( \|p_h\|_{L^2(B_R)} + \left\|\frac{L_{T,H}(t+h)-L_{T,H}(t)}{h} \phi_{T,H}(t)\right\|_{H^2(B_R)}\right)\\
		&\leq C\left( \|p_h\|_{L^2} + \epsilon_{T,H}C_R\right),
\end{split}\]
where, due to the fact that $\partial_tD$ and $\partial_tA$ tend to zero as $T$ and $H$ tend to infinity, $\epsilon_{T,H}$ is a constant which tends to zero as $T$ and $H$ tend to infinity.  Of course, using the definitions of $p_h$ and $L_{T,H}$ and integrating $L_{T,H}(t+h)p_h$ by parts yields
\[\begin{split}
	-\int p_h A \nabla p_h dyd\eta + (1 - \|D\|_\infty) \int |\nabla p_h|^2dyd\eta
		&\leq \int p_h L_{T,H}(t)p_h dyd\eta\\
		&\leq \left\|\frac{L_{T,H}(t+h)-L_{T,H}(t)}{h} \phi_{T,H}(t)\right\|_{L^2} \|p_h\|_{L^2}.
\end{split}\]
Using that $A$ tends to zero as $T$ and $H$ tend to infinity and that $\|p_h\|_{L^2} \leq C_R \|\nabla p_h\|_{L^2}$, we obtain
\[
	\|\nabla p_h\|_{L^2}
		\leq C_R \epsilon_{T,H}.
\]
Here $\epsilon_{T,H}$ is as above.  By the Poincar\'e inequality and our $H^4$ bound of $p_h$, we obtain
\[
		\|p_h\|_{H^4}
			\leq C_R \epsilon_{T,H}.
\]
Hence, $p_h$ is bounded in $H^4$ uniformly in $h$.  By Morrey's inequality, see e.g.~\cite[Theorem~7.22]{GilbargTrudinger}, $p_h$ and $\phi_{T,H}(t)$ are bounded in $C^{2,\alpha}$ for some $\alpha >0$.  Moreover, we see that the $p_h$ tends to zero in $C^{2,\alpha}$ as $T$ and $H$ tend to infinity.

We now suppose that $p_h(y_n,\eta_n)/\phi_{T,H}(t_n, y_n, \eta_n)$ as $n$ tends to infinity for some sequences $h_n \to 0$, $T_n$, $H_n$, and $(t_n,y_n,\eta_n)$.
Since $p_h$ and $\phi_{T,H}(t)$ are bounded in $C^{2,\alpha}$, we may define
\[
	p_0(y,\eta) = \lim_{n\to\infty} p_{h_n}(y,\eta),
	~~~\text{ and }~~~
	\phi(t, y,\eta)
		= \lim_{n\to\infty} \phi_{T_n,H_n}(t+t_n, y, \eta).
\]
In addition, using that $B_R$ is compact, we let $y_n \to y_\infty$ and $\eta_n \to \eta_\infty$.  Since $\|\phi_{T,H}(t)\|_{L^2(B_R)}=1$, then by compactness, $\|\phi\|_{L^\infty(B_R)} = 1$.

Since $p_h$ is bounded in $C^{2,\alpha}$, we must have that either $\phi(y_\infty,\eta_\infty) = 0$ or that $(y_\infty,\eta_\infty)\in\partial B_R$ and the normal derivative of $\phi$ is zero.  The first violates the maximum principle and the second violates the Hopf maximum principle.  This is a contradiction and thus $p_h / \phi_{T,H}$ is bounded uniformly above.

To see that $p_h / \phi_{T,H}$ tends uniformly to zero.  We may argue exactly as we did to show that $\partial_t \phi_{T,H}/\phi_{T,H}$ tends to zero, using the fact that $p_h$ tends to zero in $C^1$.  This concludes the proof.
\end{proof}

\bibliography{refs}
\bibliographystyle{plain}

\end{document}